\documentclass[11pt,reqno]{amsart}
\usepackage{amsfonts,amsmath,amssymb,amscd,amsthm}
\usepackage{hyperref, bookmark}
\usepackage[usenames,dvipsnames]{xcolor}
\usepackage{graphicx}
\usepackage{enumitem}
\usepackage[all]{xy}
\usepackage{mathtools}
\usepackage{mathrsfs}
\usepackage{xparse}
\usepackage{xpatch}

\setlist[enumerate,1]{label={\upshape(\roman*)}, leftmargin=2\parindent}
\setlist[itemize,1]{label=$\boldsymbol{\cdot}$, leftmargin=\parindent}

\usepackage{tikz}
\usetikzlibrary{arrows,backgrounds}

\newcommand{\etalchar}[1]{$^{#1}$}

\newcommand{\ie}{{\it i.e.\@ }}
\newcommand{\eg}{{\it e.g.\@ }}

\newcommand{\PaperTitle}{A reduction principle for Fourier coefficients of automorphic forms}

\hypersetup{
    pdftitle={\PaperTitle},         %
    pdfauthor={ Dmitry Gourevitch, Henrik P. A. Gustafsson, Axel Kleinschmidt, Daniel Persson, and Siddhartha Sahi },      %
    linktocpage=true,           
    colorlinks=true,            
    linkcolor=blue!75!black,    
    citecolor=green!50!black,   
    filecolor=magenta,          
    urlcolor=blue!75!black      
}

\let\oldtocsubsection=\tocsubsection
\let\oldtocsubsubsection=\tocsubsubsection
\renewcommand{\tocsubsection}[2]{\hspace{1em}\oldtocsubsection{#1}{#2}}
\renewcommand{\tocsubsubsection}[2]{\hspace{2em}\oldtocsubsubsection{#1}{#2}}

\usepackage{etoolbox}
\makeatletter
\patchcmd{\@startsection}
  {\@afterindenttrue}
  {\@afterindentfalse}
  {}{}
\makeatother

\newcommand*\IsoTo{%
  \xrightarrow[]{\raisebox{-0.5 em}{\smash{\ensuremath{\sim}}}}%
}

\newtheorem{maintheorem}{Theorem}

\theoremstyle{definition}
\newtheorem{mainalgorithm}[maintheorem]{Algorithm}

\theoremstyle{plain}

\newtheorem*{theorem*}{Theorem}
\newtheorem{lemma}{Lemma}[subsection]

\newtheorem*{conjecture*}{Conjecture}
\newtheorem{thm}[lemma]{Theorem}
\newtheorem{prop}[lemma]{Proposition}
\newtheorem{lem}[lemma]{Lemma}
\newtheorem{cor}[lemma]{Corollary}

\theoremstyle{definition}

\newtheorem{algorithm}[lemma]{Algorithm}
\newtheorem{definition}[lemma]{Definition}
\newtheorem{defn}[lemma]{Definition}

\newtheorem{notn}[lemma]{Notation}

\newtheorem{example}[lemma]{Example}

\newtheorem{notation}[lemma]{Notation}

\newtheorem*{example*}{Example}

\theoremstyle{remark}

\newtheorem*{remark*}{Remark}

\newtheorem{remark}[lemma]{Remark}

 \theoremstyle{plain}
\newcommand{\iso}{\cong}

\newcommand{\uu}{\mathcal{U}}

\newcommand{\WO}{\operatorname{WO}}

\newcommand{\bs}{\backslash}

\newcommand{\Hom}{\operatorname{Hom}}

\newcommand{\diag}{\operatorname{diag}}

\newcommand{\A}{\mathbb{A}}

\newcommand{\eps}{\varepsilon}

\newcommand{\Ker}{\operatorname{Ker}}

\newcommand{\Q}{{\mathbb Q}}
\newcommand{\R}{{\mathbb R}}

\newcommand{\C}{{\mathbb C}}
\newcommand{\K}{{\mathbb K}}

\newcommand{\Span}{{\operatorname{Span}}}

\newcommand{\Exp}{\operatorname{Exp}}

\newcommand{\Id}{\operatorname{Id}}

\newcommand{\G}{{\bf G}}

\newcommand{\alp}{{\alpha}}

\newcommand{\lam}{{\lambda}}

\newcommand{\bfG}{{\mathbf{G}}}

\newcommand{\cF}{{\mathcal{F}}}

\newcommand{\cN}{{\mathcal{N}}}
\newcommand{\cM}{{\mathcal{M}}}

\newcommand{\abs}[1]{\left|{#1}\right|}

\newcommand{\g}{{\mathfrak{g}}}
\newcommand{\fa}{{\mathfrak{a}}}

\newcommand{\fc}{{\mathfrak{c}}}
\newcommand{\fg}{{\mathfrak{g}}}
\newcommand{\fh}{{\mathfrak{h}}}
\newcommand{\fn}{{\mathfrak{n}}}

\newcommand{\fu}{{\mathfrak{u}}}
\newcommand{\fv}{{\mathfrak{v}}}
\newcommand{\fw}{{\mathfrak{w}}}

\newcommand{\cO}{{\mathcal{O}}}
\newcommand{\GL}{\operatorname{GL}}
\newcommand{\SL}{\operatorname{SL}}
\newcommand{\SO}{\operatorname{SO}}
\newcommand{\Sp}{\operatorname{Sp}}
\newcommand{\gl}{{\mathfrak{gl}}}
\newcommand{\sll}{{\mathfrak{sl}}}

\newcommand{\Sym}{\operatorname{Sym}}

\newcommand{\Lie}{\operatorname{Lie}}

\newcommand{\ad}{\operatorname{ad}}
\newcommand{\Ad}{\operatorname{Ad}}

\newcommand{\fri}{\mathfrak{i}}
\newcommand{\fp}{\mathfrak{p}}
\newcommand{\fl}{\mathfrak{l}}
\newcommand{\fm}{\mathfrak{m}}

\newcommand{\fr}{\mathfrak{r}}

\newcommand{\fz}{\mathfrak{z}}

\newcommand{\cW}{\mathcal{W}}

\renewcommand{\fri}{{\mathfrak{i}}}
\newcommand{\onto}{{\twoheadrightarrow}}

\newcommand{\into}{{\hookrightarrow}}

\newcommand{\WS}{\operatorname{WS}}

\renewcommand{\sl}{\mathfrak{sl}}

\newcommand{\cge}{\succcurlyeq}

\newcommand{\lie}[1]{\mathfrak{#1}}

\NewDocumentCommand{\intl}{g}{
    \IfNoValueTF{#1}{\int\limits}{\int\limits_{\mathclap{#1}}}
}
\NewDocumentCommand{\suml}{g}{
    \IfNoValueTF{#1}{\sum\limits}{\sum\limits_{\mathclap{#1}}}
}

\newcommand{\Oh}{\mathcal{O}}

\numberwithin{equation}{section}

\begin{document}

\author[D. Gourevitch]{Dmitry Gourevitch}
\address{Dmitry Gourevitch,
Faculty of Mathematics and Computer Science,
Weizmann Institute of Science,
POB 26, Rehovot 76100, Israel }
\email{dmitry.gourevitch@weizmann.ac.il}
\urladdr{\url{http://www.wisdom.weizmann.ac.il/~dimagur}}

\author[H. Gustafsson]{Henrik P. A. Gustafsson}
\address{\hspace{-\parindent}Henrik Gustafsson \newline 
  \hspace{-\parindent}\textnormal{Until September 15, 2019:} \newline
  \indent Department of Mathematics, Stanford University, Stanford, CA 94305-2125. \newline 
  \textnormal{Since September 16, 2019:} \newline
  \indent School of Mathematics, Institute for Advanced Study, Princeton, NJ~08540. \newline 
  \indent Department of Mathematics, Rutgers University, Piscataway, NJ~08854. \newline 
  \indent Department of Mathematical Sciences, University of Gothenburg and Chalmers University of Technology, SE-412~96 Gothenburg, Sweden.}
\email{henrik.pa.gustafsson@gu.se}
\urladdr{\url{http://hgustafsson.se}}

\author[A. Kleinschmidt]{Axel Kleinschmidt}
\address{Axel Kleinschmidt,  {\it Max-Planck-Institut f\"{u}r Gravitationsphysik (Albert-Einstein-Institut)},
Am M\"{u}hlenberg 1, DE-14476 Potsdam, Germany, and 
{\it International Solvay Institutes}, 
ULB-Campus Plaine CP231, BE-1050, Brussels, Belgium}
\email{axel.kleinschmidt@aei.mpg.de}

\author[D. Persson]{Daniel Persson} 
\address{Daniel Persson, Chalmers University of Technology, Department of Mathematical Sciences\\
SE-412\,96 Gothenburg, Sweden}
\email{daniel.persson@chalmers.se}

\author[S. Sahi]{Siddhartha Sahi}
\address{Siddhartha Sahi, Department of Mathematics, Rutgers University, Hill Center -
Busch Campus, 110 Frelinghuysen Road Piscataway, NJ 08854-8019, USA}
\email{sahi@math.rugers.edu}



\title{\expandafter\MakeUppercase\expandafter{\PaperTitle}}
\maketitle
\begin{abstract}
We consider a general class of Fourier coefficients for an automorphic form on a finite cover of a reductive adelic group ${\bf G}(\A_\K)$, associated to the data of a `Whittaker pair'. We describe a quasi-order on Fourier coefficients, and an algorithm that gives an explicit formula for any coefficient in terms of integrals and sums involving higher coefficients. 

The maximal elements for the quasi-order are `Levi-distinguished' Fourier coefficients, which correspond to taking the constant term along the unipotent radical of a parabolic subgroup, and then further taking a Fourier coefficient with respect to a $\K$-distinguished nilpotent orbit in the Levi quotient. Thus one can express any Fourier coefficient, including the form itself, in terms of higher Levi-distinguished coefficients.

In follow-up papers we use this result to determine explicit Fourier expansions of minimal and next-to-minimal automorphic forms on split simply-laced reductive groups, and to obtain Euler product decompositions of their top Fourier coefficients.

\end{abstract}

\pagebreak 
\tableofcontents
\pagebreak 

\section{Introduction}
\subsection{Main results}
\label{intro}

In this paper we establish a reduction principle for unipotent periods of an automorphic form on a reductive algebraic group $\bf G$ defined over a number field $\K$. Such periods generalize Fourier coefficients for classical modular forms, and are important because they often encode quantities of arithmetic interest. Our motivation comes in part from string theory, where certain so-called non-perturbative effects can be expressed in terms of unipotent periods~\cite{Green:1997tv,Pioline:2010kb,Green:2011vz}. 
This 
is explored further in a companion paper \cite{Part2} that uses the results of the present paper in an essential way.

We work in the useful generality of a pair $(G,\Gamma),$ where $G$ is a
finite central extension of the adele group $\G(\A) = \G(\A_\K)$, and $\Gamma \subset G$ is a discrete subgroup  on which the covering map restricts to an isomorphism with ${\bf G}(\K)$. This class includes the important central extensions defined in \cite{BryDel}. Also in this case, by \cite[Appendix I]{MWCov}, the covering map has a canonical splitting over every unipotent subgroup $N = {\bf N}(\A)$, which we will use to identify $N$ as a subgroup of $G$.  If $\eta$ is a left $\Gamma$-invariant function on $G$, and $\chi$ is a character of $N$ that is trivial on $N\cap \Gamma$, then the $(N,\chi)$-unipotent period of $\eta$ is the integral
\begin{equation}
    \uu_{(N,\chi)}[\eta](g)=\int_{[N]}\eta(ng)\chi(n)^{-1}dn, \quad [N]:=({N\cap \Gamma})\backslash N.
\end{equation}
Since $[N]$ is compact, the period integral converges under mild conditions on $\eta$; although it is quite unlikely that one can say anything useful in such generality.

In this paper we consider a special class of unipotent periods $\cF_{S,\varphi}$ associated to certain pairs $(S,\varphi)\in \fg\times \fg^*$, where $\fg=\fg(\K)$ is the Lie algebra of $\bfG(\K)$. We refer to these as $(S,\varphi)$\emph{-Fourier coefficients}, or simply as Fourier coefficients. As we explain in next section, this class includes the unipotent periods studied in the theory of automorphic forms \cite{GRS2,GRS,Ginz,GH,JLS}, which we call \emph{neutral} Fourier coefficients, and also those arising in string theory~\cite{Pioline:2010kb,Green:2011vz,FGKP}, which we call \emph{parabolic} Fourier coefficients. 

Our main result is a reduction algorithm that allows us to express a given Fourier coefficient in terms of others, which are \emph{higher} with respect to a natural quasi-order $\cge$. This is the more important direction for applications. It is much easier to write a Fourier coefficient in terms of \emph{lower} coefficients; such expressions tend to be simpler but less useful -- see Proposition \ref{prop:domin} and its proof below, which uses the techniques of \cite{GGS}.

We now give a quick sketch of the essential ideas. We say $S\in \fg$ is $\Q$-semisimple if $\ad(S)$ acts semisimply on $\fg$ with eigenvalues in $\Q$, and we write $\fg^S_{\lam}$ for the $\lambda$-eigenspace. We set  
$$
\fg^S_{>\lam}=\textstyle{\bigoplus\nolimits}_{\mu>\lam}\fg^S_{\mu},\quad \fg^S_{\geq\lam}=\textstyle{\bigoplus\nolimits}_{\mu\geq\lam}\fg^S_{\mu}, \quad \text{ \it {etc.}},
$$
with similar notation for $\fg^*$. We say $(S,\varphi)\in \fg\times \fg^*$ is a \emph{Whittaker pair} if $S$ is $\Q$-semisimple and $\varphi\in(\fg^*)^S_{-2}$. In this case $\varphi$ is given by the Killing form pairing with a unique nilpotent element $f_\varphi$ in $\fg^S_{-2}$.
We write $\fg_{\varphi} \subset \fg$ for the stabilizer of $\varphi$, and we set
\begin{equation}
        \fn_{S,\varphi} = \fg^S_{> 1}\oplus (\fg^S_{1}\cap \fg_{\varphi}).
\end{equation}
Then $ \fn_{S,\varphi}$ is a nilpotent Lie algebra and $\varphi\vert_{\fn_{S,\varphi}}$ is a Lie algebra character. By exponentiation we get a unipotent subgroup $N_{S,\varphi}\subset G$ and a character $\chi_{S,\varphi}$, and we define
\begin{equation}
  \cF_{S,\varphi}[\eta]=\uu_{(N_{S,\varphi},\chi_{S,\varphi})}[\eta].
\end{equation}

We now define a $\Gamma$-invariant quasi-order on Whittaker pairs.
If $\varpi=(H,\varphi)$ and $\varpi'=(H',\varphi')$ are Whittaker pairs, then we write $\varpi\cge\varpi'$, first in the following two basic cases:
\begin{enumerate}
\item $ \varphi'=\varphi$, $H$ and $H'$ commute, and we have 
\begin{align}
\fg_{\varphi}\cap \fg^H_{\geq 1} \subseteq \fg^{H-H'}_{\geq 0}
\end{align}
\item $H=H'$, and there is a parabolic subalgebra $\fp=\fl+\fn$ defined over $\K$ such that 
\begin{align}\label{=IntOrder}
H, f_{\varphi'} \in \fl,\quad f_\varphi-f_{\varphi'} \in \fn.
\end{align}
\end{enumerate}
More generally, we write $\varpi\cge\varpi'$ if there is a sequence
$\varpi_1,\ldots,\varpi_n$ with $\varpi_1=\varpi$, $\varpi_n=\varpi'$, 
and elements $\gamma_1,\ldots,\gamma_{n-1}$ in $\Gamma$ such that  $\varpi_i \cge \gamma_i\cdot\varpi_{i+1}$ in the sense of (i) or (ii).  

We will also consider separately the quasi-orders given by (i) and (ii). Thus we will write
 \begin{equation}
 H\cge_\varphi H' \text{ if } (H,\varphi) \cge(H',\varphi);\quad \text{ and }\varphi \cge_H \varphi' \text{ if } (H,\varphi) \cge(H,\varphi').
 \end{equation}
 
We say that a pair $(H,\varphi)$ is \emph{Levi-distinguished }if $H$ is maximal with respect to $\cge_\varphi$. In this case, $\cF_{H,\varphi}$ reduces to a $\mathbb{K}$-distinguished coefficient on a Levi quotient of a parabolic subgroup, by first taking the constant term along the unipotent radical -- see \S\ref{subsec:Levi-DistPrel} for details.
 
{\bf Algorithm \ref{alg:domin}}, our main result, described in \S\ref{sec:levi-distinguished}, relates the Fourier coefficients for two pairs $(S,\varphi)\cge (H,\varphi)$ and gives an explicit formula of the form
\begin{equation}\label{Alg:A}
\cF_{H,\varphi}[\eta]= \cM_H^S(\cF_{S,\varphi}[\eta]) + \text{higher terms}.
\end{equation}
Here the ``main'' term $\cM_H^S$ is a discrete sum of integral transforms that we describe shortly. The ``higher'' terms are similar expressions, described in \S\ref{sec:levi-distinguished}, involving $\cF_{H',\varphi'}$ such that ${(H',\varphi')} \cge {(H,\varphi)}$ but $\varphi'$ is \emph{not} in the $\Gamma$-orbit of $\varphi$, which implies that the closure of the $\G(\C)$-orbit of $\varphi'$ properly contains that of $\varphi$. By iteration, we deduce that
\begin{enumerate}
\item[(a)] any $\cF_{H,\varphi}$ can be expressed via Levi-distinguished $\cF_{S,\psi}$ satisfying $(S,\psi)\cge (H,\varphi)$.
\item[(b)] in particular, $\eta=\cF_{0,0}[\eta]$ can be expressed in terms of Levi-distinguished $\cF_{S,\psi}$.
\item[(c)] if $\eta$ is cuspidal (see \S\ref{subsec:LeviDist})
it can be expressed in terms of $\K$-distinguished $\mathcal{F}_{S,\psi}$.
\end{enumerate}

We now define $\cM_H^S$. For this we write $(\mathfrak{g}^H_{>\lambda})^S_\mu=\fg^H_{>\lam}\cap \fg^S_\mu$ etc., and we set 
  \begin{equation}
  \fu:= (\fg^{H}_{>1})^S_1, \quad
  \fv:= (\fg_{> 1}^{H})^S_{<1}, \quad
  \fw:=(\fg^{H}_{1})^S_{<1}.
  \end{equation}
We regard these as nilpotent subquotients of the Lie algebra $\fg$ via the identifications
\begin{equation}
  \fu\cong(\fg^{H}_{>1})^{S}_{\geq1}/(\fg^{H}_{>1})^{S}_{>1}, \;
  \fv\cong\fg^H_{> 1}/(\fg^H_{> 1})^{S}_{\geq 1}, \;\fw\cong\fg^H_{\geq 1}/\left(\fg^H_{> 1}+ (\fg^H_{\geq 1})^S_{\geq 1}\right),
  \end{equation}
 and we define corresponding subquotients of the group $G$ as follows:
 \begin{equation}
 \label{uvw}
U=\Exp(\fu(\A)),\quad V=\Exp(\fv(\A)),\quad \Omega=\Exp(\fw(\K)).
  \end{equation} 
We note that $\Omega$ is a discrete group since $\fw=\fw(\K)$ is a lattice. We now set
    \begin{equation}\label{=HfromS}
             \cM_H^S(\cF_{S,\varphi}[\eta]) = \sum_{w\in \Omega}\,
            \int_{V} \int_{ {[ U]}} \cF_{S,\varphi}[\eta](wvug) \, dudv \, .
        \end{equation}

There is an important special case where the higher terms in (\ref{Alg:A}) vanish. We write $\psi \succ \varphi$ if $\psi \cge_H \varphi$ for some $H$, but $\psi \not\in \Gamma\cdot\varphi$, and we say $\varphi$ is in the \emph{Whittaker support }${\rm WS}(\eta)$ of $\eta$ if $\cF_{S,\varphi}[\eta]\ne 0$ for some $S$, but if $\psi \succ \varphi$ then $ \cF_{S,\psi}[\eta]= 0$ for all $S$. 

\setcounter{maintheorem}{1}  
\begin{maintheorem}[see \S\ref{subsec:Pfs}]
\label{thm:IntTrans}
If $\varphi$ is in ${\rm WS}(\eta)$ and $(S,\varphi)\cge (H,\varphi)$ then 
            $\cF_{H,\varphi}[\eta] =\cM_H^S(\cF_{S,\varphi}[\eta]).$            
\end{maintheorem}

\subsection{Classes of Fourier coefficients and further results}\label{sec:1.2}
The most natural way to complete a nilpotent $\varphi \in \fg^*$ to a Whittaker pair is provided by the Jacobson--Morozov theorem. Namely, $\varphi$ is given by the Killing pairing with a unique nilpotent $f\in \fg,$ and $f$ in turn can be completed to an $\mathfrak{sl_2}$-triple $(e,h,f)$. Then $(h,\varphi)$ is a Whittaker pair. We will call such pairs, and the corresponding coefficients, \emph{neutral}. Such  coefficients were extensively studied in \cite{GRS2,GRS,Ginz,GH,JLS}, where the name ``Fourier coefficient" was reserved exclusively for this case. 

Many applications use the class of  \emph{parabolic} Fourier coefficients, which are period integrals with respect to a character of the unipotent radical $U$ of a parabolic subgroup $P=LU$. Despite their prevalence in the literature, these have in general not been evaluated in terms of known functions on $\mathbf{G}(\mathbb{A})$. 

An important subclass of parabolic Fourier coefficients are those with respect to maximal unipotent subgroups, which we call \emph{Whittaker coefficients}. When the automorphic form is spherical and the unipotent character is non-degenerate these factorize over primes and the non-archimedean factors can be explicitly computed using the  Casselman--Shalika formula~\cite{CasselmanShalika}. By Lemma \ref{lem:WhitPL} below these are precisely the Levi-distinguished Fourier coefficients corresponding to nilpotent orbits that are principal in Levi subalgebras defined over $\K$. We will call such orbits \emph{$PL$-orbits} (see \S \ref{subsec:defPL} below for more details).

It follows from \cite{GGS} that for any Whittaker pair $(H,\varphi)$ there exists $h\in \fg$ such that $(h,\varphi)$ is a neutral pair and $H\cge_\varphi h$ (see Corollary \ref{cor:domin-neutral} below). 
Moreover, by \cite[Theorem C]{GGS}, if $\cF_{h,\varphi}[\eta]\equiv 0$ for some neutral pair $(h,\varphi)$ then $\cF_{S,\varphi}[\eta]\equiv 0$ for any Whittaker pair $(S,\varphi)$ with the same nilpotent element $\varphi$. Thus, neutral coefficients can be used in the definition of the Whittaker support.
More generally, if $S\cge_\varphi H$ then $\cF_{S,\varphi}$ can be obtained from $\cF_{H,\varphi}$ by an integral transform (that can also involve discrete summation), see Proposition \ref{prop:domin} below.  
In Lemma \ref{lem:dim} below we show that 
Levi-distinguished Whittaker pairs have maximal dimension of $\fn_{S,\varphi}$ among all Whittaker pairs with the same $\varphi$.

As an illustrative example we may take the coefficient $B=\cF_{S,\varphi}$ to be a Whittaker
coefficient with a character supported only on the exponentiated root space of a single
simple root, and the coefficient $C=\cF_{H,\varphi}$ to be a Fourier coefficient with respect to the
unipotent radical of the standard maximal parabolic subgroup obtained from the same root,
together with a restriction of the same character. Then, $B$ can be obtained as a period integral of $C$ over the quotient of the unipotent
groups. Theorem~\ref{thm:IntTrans} and Algorithm~\ref{alg:domin} allow us to go in the opposite direction and obtain $C$ from $B$ and from Levi-distinguished coefficients corresponding to higher orbits.

By Lemma \ref{lem:WhitPL} below,  if $\varphi$ is a principal nilpotent in $L$  then $\cF_{H+Z,\varphi}$ is a Whittaker coefficient. 
By Corollary \ref{cor:LeviDistDomin} below, for any Whittaker pair $(H,\varphi)$, there exists an $S$ such that $(S,\varphi)$ is a Levi-distinguished Whittaker pair, and $S\succcurlyeq_\varphi H$. 
We refer to \S \ref{subsec:Levi-DistPrel} below for more details.
To summarize, for any orbit $\Oh$ we have that
\begin{equation}
  \text{neutral} \; \preccurlyeq \; \text{any} \; \preccurlyeq \; \text{Levi-distinguished} \; \supseteq \text{Whittaker}.
\end{equation}
 
In particular, for the zero orbit $\{0\}$, the neutral coefficient is the identity map $\cF_{0,0}[\eta]=\eta$, while the Whittaker coefficient is the constant term map. From this summary we conclude that if all orbits $\cO\in \WO(\eta)$ are PL-orbits
then Algorithm \ref{alg:domin} allows us to express $\eta$, and all its Fourier coefficients, through its Whittaker coefficients.

In the case $G=\GL_n(\A)$ this generalizes the classical result by Piatetski-
Shapiro and Shalika \cite{PiatetskiShapiro,Shalika} that expresses every cuspidal form through its Whittaker coefficients with respect to non-degenerate characters. We explain how Algorithm \ref{alg:domin} allows us to reproduce this result in \S \ref{subsec:GL} below.

\subsection{Related works}

Our main tool is a generalization of the deformation technique of \cite{GGS,GGS:support}, which in turn builds on the root-exchange method of \cite{GRS2,GRS}. 

There are three crucial differences between the approach in the current paper and that of \cite{GGS:support}: we consider automorphic forms (rather than mostly local representations), we consider the general case of the relation $S\cge_{\varphi} H$ (rather than requiring $H$ to be neutral), and we give explicit formulas relating various Fourier coefficients, while \cite{GGS:support} concentrates on vanishing properties.
A special case of Algorithm \ref{alg:domin} in which $(H,\varphi)$ is neutral can be established using the technique of \cite{GGS:support}. One could combine this with a similar, but easier, version (Proposition \ref{prop:domin} below) to relate any $\cF_{H,\varphi}$ to any $\cF_{S,\varphi}$. However this would be less useful than Algorithm \ref{alg:domin}, which provides more compact expressions since it proceeds directly from $(H,\varphi)$ to $(S,\varphi)$ without a detour through neutral coefficients.

One can show that for split simply-laced groups the so-called minimal and the next-to-minimal orbits are always PL. Thus Algorithm \ref{alg:domin} and Theorem~\ref{thm:IntTrans} give formulas for automorphic forms attached to the minimal and next-to-minimal representations of simply-laced groups, as well as all their Fourier coefficients, in terms of their Whittaker coefficients.  
We develop these formulas in a companion paper~\cite{Part2}.
Another application of Theorem~\ref{thm:IntTrans} is to deduce  that certain Fourier coefficients are Eulerian~\cite{Eulerianity}.

A Fourier expansion for the discrete spectrum of $\GL_n(\A)$ is provided in \cite{JiangLiu}, generalizing \cite{PiatetskiShapiro,Shalika}. In \S \ref{subsec:GL} below we apply  Algorithm \ref{alg:domin} to provide a further generalization.

\subsection{Structure of the paper}
In \S \ref{sec:prel} we give the definitions of the notions mentioned above, as well as of \emph{Whittaker triples} and \emph{quasi-Fourier coefficients.}
These are technical notions defined in \cite{GGS:support} and widely used in the current paper as well. 

In \S \ref{sec:GenLemmas} we relate Fourier and quasi-Fourier coefficients corresponding to different Whittaker pairs and triples. To do that we further develop the deformation technique of \cite{GGS,GGS:support}, making it both more general, more explicit, and better adapted to the global case.

\S \ref{sec:levi-distinguished} contains the main results of the paper.
In \S \ref{subsec:alg} we describe Algorithm \ref{alg:domin}. 
The algorithm proceeds by deforming a Whittaker pair $(H,\varphi)$ to a bigger pair $(S,\varphi)$ along a straight line $H+t(S-H)$ with $t\in [0,1]$ in the Cartan subalgebra. At certain critical values $t$, additional quasi-Fourier coefficients (associated with higher orbits) are generated. These can be rewritten in terms of higher Fourier coefficients by proceeding in a straight line away from a neutral element, see Algorithm~\ref{alg:quasi}. The final result consists of the main term for $(S,\varphi)$ together with these higher Fourier coefficients. 
In \S \ref{subsec:Pfs} we prove that the algorithm is correct and terminates in a finite number of steps. We also derive Theorem \ref{thm:IntTrans}. In \S \ref{subsec:more} we prove additional results, including a formula for $\cF_{S,\varphi}$ in terms of $\cF_{H,\varphi}$ (this direction is opposite to that of Theorem \ref{thm:IntTrans}). In \S \ref{subsec:LeviDist} we give a constructive proof that 
for any Whittaker pair $(H,\varphi)$, there exists an $S$ such that $(S,\varphi)$ is a Levi-distinguished Whittaker pair, and $S\succcurlyeq_\varphi H$.

In \S \ref{sec:examples} we provide explicit examples in the cases $\SL_4, \GL_n, \Sp_4$ and Heisenberg parabolics of arbitrary simply-laced Lie groups.

Two appendices contain proofs of geometric lemmas on PL-orbits and on our order relation on rational nilpotent orbits.

\subsection{Acknowledgements} 
The authors are grateful for helpful discussions with Ben Brubaker, David Ginzburg and Stephen D. Miller. 
We are particularly thankful to Joseph Hundley for sharing with us his insights  on nilpotent orbits for exceptional groups.
We wish to thank the anonymous referees for very useful comments on an earlier version of this paper.
We also thank the Banff International Research Station for Mathematical Innovation and Discovery and the Simons Center for Geometry and Physics for their hospitality during different stages of this project.

D.G.\ was partially supported by ERC StG grant 637912 and BSF grant 2019724. During his time at Stanford University, H.G.\ was supported by the Knut and Alice Wallenberg Foundation.
Later, H.G.\ was supported by the Swedish Research Council (Vetenskapsr\aa det), grant no.\ 2018-06774. S.S.\ was partially supported by NSF grants DMS-1939600 and DMS-2001537, and Simons' foundation grant 509766. D.P.\ was supported by the Swedish Research Council (Vetenskapsr\aa det), grant no.\ 2018-04760.

\section{Preliminaries on nilpotent orbits and Whittaker pairs}\label{sec:prel}
\setcounter{lemma}{0}

In this section, we fix  basic notation for Whittaker pairs and Fourier coefficients. We also introduce some new preparatory results.

\subsection{Whittaker pairs}
As in the introduction, let $\K$ be a number field, $\mathfrak{o}$ its ring of integers, and let $\A=\A_{\K}$ be its ring of adeles. Let $\mathbb{T} = \{z \in \mathbb{C} : \abs{z} = 1\}$ and fix a non-trivial additive character $\chi : \A \to \mathbb{T}$, which is trivial on $\K$. Then $\chi$ defines an isomorphism between $\A$ and the character group $\hat{\A} := \Hom(\mathbb{A}, \mathbb{T})$ via the map $a\mapsto \chi_{a}$, where $\chi_{a}(b)=\chi(ab)$ for $a,b\in \A$. Furthermore, this isomorphism restricts to an  isomorphism
\begin{equation}\label{eq:chi_isomorphism}
  \widehat{\A/\K} := \Hom(\mathbb{A}/\mathbb{K}, \mathbb{T}) \cong \{r\in \hat{\A}\, : \,r|_{\K}\equiv 1\}=\{\chi_{a} : a \in \K\}\cong \K \, ,
\end{equation}
which means that we may parametrize characters on $\mathbb{A}$ trivial on $\mathbb{K}$ by elements in $\mathbb{K}$.

Let ${\bf G}$ be a reductive group defined over $\K$, ${\bf G}(\mathbb{A})$ the group of adelic points of ${\bf G}$ and let $G$ be a finite central extension of ${\bf G}(\A)$. That is,
\begin{equation}
  1 \to C \to G \xrightarrow{\operatorname{pr}} \mathbf{G}(\mathbb{A}) \to 1
\end{equation}
for some finite group $C$.

We assume that there exists a section ${\bf G}(\K)\to G$ of the projection $\operatorname{pr} : G \onto {\bf G}(\A)$.
Fix such a section and denote its image in $G$ by $\Gamma$. By \cite[Appendix I]{MWCov}, the cover $G\onto {\bf G}(\A)$ canonically splits over unipotent subgroups, and thus we will consider such subgroups as subgroups of $G$.
Let $\fg(\K)$ denote the Lie algebra of $\G(\K) \iso \Gamma$ which we will often abbreviate to $\mathfrak{g}$. 
Let $\mathfrak{v}$ be a nilpotent subalgebra of $\fg$ and let $\fv(\A):=\fv\otimes_{\K}\A$ be its adelization.
As in the introduction, we denote by $\Exp(\fv)$ the unipotent subgroup of $\Gamma$ obtained by exponentiation of $\fv$ using the above split over unipotent subgroups, and we denote by $V:=\Exp(\fv(\A))$ the unipotent subgroup of $G$ obtained by exponentiation of $\fv(\A)$.
We note that $\Exp(\mathfrak{v}) = V \cap \Gamma$ and for later convenience we will denote by $[V]$ the quotient $(V\cap \Gamma)\backslash V$.

To conveniently describe different unipotent subgroups of $G$ and characters on these subgroups we introduce the following notion.

\begin{definition}
\label{def:pair}
A \emph{Whittaker pair} is an ordered pair $(S,\varphi)\in \fg\times \fg^*$ such that $S$ is a rational semi-simple element (that is, a semi-simple element for which the eigenvalues of the adjoint action are in $\Q$), and $\ad^*(S)\varphi=-2\varphi$. 
\end{definition}

We will often identify $\varphi \in \fg^*$ with its dual nilpotent element $f = f_\varphi \in \lie g$ with respect to the Killing form $\langle \, , \, \rangle$.
We will say that $\varphi$ is \emph{nilpotent} if $f_\varphi$ is a nilpotent element of $\fg$. Equivalently, $\varphi\in \fg^*$ is nilpotent if and only if the Zariski closure of its coadjoint orbit includes zero. For example, if $(S,\varphi)$ is a Whittaker pair then $\varphi$ is nilpotent.

For any rational semi-simple $S\in \fg$ and $\lambda\in \Q$ we introduce the following notation 
\begin{equation}
  \fg_{\lambda}^{S}:=\{X\in \fg \, : \, [S,X]=\lambda X\}, \qquad
  \fg_{> \lambda}^{S}:=\bigoplus_{\substack{\mu\in \Q \\ \mu>\lambda}}\fg_{\mu}^S, \qquad
  \fg_{\geq \lambda}^S :=\fg_{\lambda}^{S}\oplus\fg_{>\lambda}^S\, ,
\end{equation}
and analogously for $\fg_{< \lambda}^{S}$ and $\fg_{\leq \lambda}^{S}$, with a similar use of notation for $\fg^*$.

For any $\varphi\in \fg^*$  let $\mathfrak{g}_\varphi$ be the centralizer of $\varphi$ in $\mathfrak{g}$ under the coadjoint action and define an anti-symmetric form $\omega_{\varphi} : \mathfrak{g} \times \mathfrak{g} \to \mathbb{K}$  by $\omega_{\varphi}(X,Y)=\varphi([X,Y])$.
We extend $\varphi$ and $\omega_\varphi$ to a functional and an anti-symmetric form on $\mathfrak{g}(\mathbb{A})$ respectively by linearity.
Given a Whittaker pair $(S,\varphi) \in \g \times \g^*$, we let $\mathfrak{u} := \mathfrak{g}^S_{\geq 1}$ and define 
\begin{equation}\label{=Nsphi}
  \fn_{S,\varphi}:=\{ X \in \mathfrak{u} \,:\, \omega_\varphi(X,Y)=0 \,\,\textrm{for all $Y\in \mathfrak{u}$}\}
  \quad \text{and} \quad
  N_{S,\varphi}:=\Exp \bigl( \fn_{S,\varphi}(\A) \bigr)
\end{equation} 
which, by Lemma~\ref{lem:help} below, can also be written as 
\begin{equation}
    \label{eq:N_Sphi}
    \lie n_{S,\varphi} = \lie g^S_{>1} \oplus (\lie g^S_1 \cap \lie g_\varphi) .
\end{equation}
Note that $\fn_{S,\varphi}$ is an ideal in $\mathfrak{u}$ with abelian quotient, and that $\varphi$ defines a character of~$\fn_{S,\varphi}$. 

We define a corresponding character $\chi_\varphi$ on $N_{S,\varphi}$, trivial on $N_{S,\varphi} \cap \Gamma$, by
\begin{equation}
  \chi_\varphi(n) := \chi(\varphi(\log n)) = \chi(\langle f_\varphi, \log n\rangle) \, .
\end{equation}

More generally, let $\fr\subseteq \fu$ be any isotropic subspace (not necessarily maximal) with respect to $\omega_\varphi|_{\fu}$, that includes $\lie n_{S,\varphi}$. Note that $\fn_{S,\varphi} \subseteq \fr \subseteq \fu$, and that $\fn_{S,\varphi}$ and $\fr$ are ideals in $\fu$. Let $R=\Exp \bigl(\fr(\A)\bigr)$. 
Then $\chi_{\varphi}^R : R \to \mathbb{T}$ defined by $\chi_\varphi^R(r)=\chi(\varphi(\log r))$ is a character of $R$ trivial on $R \cap \Gamma$.
Indeed, since $\mathfrak{r}$ is isotropic, we have that $\omega_\varphi|_{\mathfrak{r(\mathbb{A})}} = 0$ and thus $\chi_\varphi^R \in \Hom(R, \mathbb{T})$, and $\varphi(X) \in \mathbb{K}$ for $X \in \mathfrak{r}(\mathbb{K})$.

\begin{definition}
  We call a function on $G$ an \emph{automorphic function} if it satisfies the following properties:
  \begin{enumerate}
    \item invariant under the left action of $\Gamma$,
    \item finite under the right action of the preimage in $G$ of $\prod_{\text{finite }\nu}{\bf G}(\mathfrak{o}_\nu),$ and
    \item smooth when restricted to the preimage in $G$ of $\prod_{\text{infinite }\nu}{\bf G}(\K_\nu)$. 
  \end{enumerate}
  We denote the space of all automorphic functions by $C^{\infty}(\Gamma\backslash G)$.
\end{definition}

\begin{definition}\label{def:Whit}
    Let $(S,\varphi)$ be a Whittaker pair for $\g$ and let $R, N_{S,\varphi},\chi_{\varphi}$ and $\chi_{\varphi}^{R}$ be as above. For an automorphic function $\eta$, we define the \emph{Fourier coefficient} of $\eta$ with respect to the pair $(S,\varphi)$ to be 
    \begin{equation}\label{eq:Whittaker-Fourier-coefficient}
        \cF_{S,\varphi}[\eta](g):=\intl_{[N_{S,\varphi}]}\eta(ng)\, \chi_{\varphi}(n)^{-1}\, dn.
    \end{equation}
    We also define its \emph{$R$-Fourier coefficient} to be
    the function 
    \begin{equation}\label{eq:extended_Whittaker-Fourier_coefficient}
        \cF_{S,\varphi}^{R}[\eta](g):=\intl_{[R]}\eta(rg)\, \chi_{\varphi}^{R}(r)^{-1}\, dr.
    \end{equation}
    Observe that if $\pi$ denotes a subrepresentation of $C^{\infty}(\Gamma \backslash G)$ that contains $\eta$ then $\cF_{S,\varphi}[\eta]$ and $\cF_{S,\varphi}^{R}[\eta]$ 
    are matrix coefficients corresponding to the vector $\eta\in \pi$ and the functional on the space of automorphic functions defined by the integrals above.
\end{definition}

Note that $\mathcal{F}_{0,0}[\eta]=\eta$. For a general Whittaker pair $(S,\varphi),$ 
$\mathcal{F}_{S,\varphi}[\eta](g)$ is a smooth function on $G$ in the above sense, but is not invariant under $\Gamma$ any more. On the other hand, its restriction to the joint centralizer $G_{S,\varphi}$ of $S$ and $\varphi$ is left $G_{S,\varphi}\cap \Gamma$-invariant. As shown in \cite{GH}, if $\eta$ is also $\fz$-finite and has moderate growth, then the restriction of $\eta$ to $G_{S,\varphi}$ still has moderate growth, but may stop being $\fz$-finite.

\begin{remark}
In \cite[\S 6]{GGS} the integrals~\eqref{eq:Whittaker-Fourier-coefficient} and~\eqref{eq:extended_Whittaker-Fourier_coefficient} above are called Whittaker--Fourier coefficients, but in this paper we call them Fourier coefficients for short.
\end{remark}

\begin{definition}
  A Whittaker pair $(h,\varphi)$ is called \emph{neutral} if either $(h,\varphi)=(0,0)$, or $h$ and the Killing form pairing $f = f_\varphi$ with $\varphi$ can be completed to an $\sl_2$-triple $(e,h,f)$. Equivalently, $(h, \varphi)$ is called neutral if the map $X \mapsto \ad^*(X)\varphi$ defines an epimorphism $\fg^h_{0}\onto(\fg^*)^h_{-2} $, and $h$ can be completed to an $\sl_2$-triple.
  For more details on $\sl_2$-triples over arbitrary fields of characteristic zero see \cite[\S 11]{Bou}. 
\end{definition}

\begin{definition}
We say that $(S,\varphi,\varphi')$ is a \emph{Whittaker triple} if $(S,\varphi)$ is a Whittaker pair and $\varphi'\in (\fg^*)^{S}_{>-2}$.
\end{definition}

For a Whittaker triple $(S,\varphi,\varphi')$, let $U,R,$ and $N_{S,\varphi}$ be as in Definition \ref{def:Whit}. Note that $\varphi+\varphi'$ defines a  character of $\fr$. Extend it by linearity to a character of $\fr(\A)$ and define an automorphic character $\chi_{\varphi+\varphi'}$ of $R$ by $\chi_{\varphi+\varphi'}^{R}(\exp X):=\chi(\varphi(X)+\varphi'(X))$.
For an example for this notation see \S \ref{subsec:Whit3} below.

\begin{definition} For an automorphic function $f$, we define its \emph{$(S,\varphi,\varphi')$-quasi Fourier coefficient} to be the  function
\begin{equation}\label{eq:Quasi-Whittaker-coefficient}
 \cF_{S,\varphi,\varphi'}[\eta](g):=\intl_{[N_{S,\varphi}]}\chi_{\varphi+\varphi'}(n)^{-1}\eta(ng)dn.
\end{equation}
We also define its \emph{$(S,\varphi,\varphi',R)$-quasi Fourier coefficient} to be
the  function
 \begin{equation}\label{eq:quasi_Whittaker}
\cF_{S,\varphi,\varphi'}^{R}[\eta](g):=\intl_{[R]}\chi_{\varphi+\varphi'}^{R}(r)^{-1}\eta(rg)dr.
 \end{equation}
 \end{definition}

\begin{definition}
We call a $\K$-subgroup of $\bf G$ a \emph{split torus of rank $m$} if it is isomorphic as a $\K$-subgroup to $\GL_1^m$. We call a Lie subalgebra $\fl \subseteq \fg$ a $\K$-Levi subalgebra if it is the centralizer of a split torus.
\end{definition}

\begin{remark}
We note that the Lie algebra of any split torus is spanned by rational semisimple elements. Consequently, a subalgebra of $\fl\subseteq \fg$ is a $\K$-Levi subalgebra if and only if it is the centralizer of a rational semisimple element of $\fg$. Another equivalent condition is that $\fl$ is the Lie algebra of a Levi subgroup of a parabolic subgroup of $G$ defined over $\K$.
\end{remark}

For convenience, we fix a complex embedding  $\sigma:\K\into\C$, which allows us to map a $\Gamma$-orbit $\Oh$ in $\mathfrak{g}$ to a $\mathbf{G}(\mathbb{C})$-orbit in $\mathfrak{g}(\mathbb{C}) := \mathfrak{g} \otimes_{\sigma(\mathbb{K})} \mathbb{C}$.
One can show, using \cite{Dok}, that the complex orbit corresponding to $\cO$ does not depend on $\sigma$. However, we will not need this fact.

\begin{defn}\label{def:dominate}
    Let $(H,\varphi)$ and $(S,\varphi)$ be two Whittaker pairs with the same $\varphi$. We say that $(H,\varphi)$ \emph{dominates} $(S,\varphi)$ if $H$ and $S$ commute and 
\begin{equation}\label{=domin}
        \fg_{\varphi}\cap \fg^H_{\geq 1}\subseteq \fg^{S-H}_{\geq0}\,.
    \end{equation}
\end{defn}

This relation is denoted $S \cge_{\varphi} H$ in the introduction.

\subsection{Principal nilpotent elements, PL elements and standard Whittaker pairs}
\label{subsec:defPL}

We introduce some notions for coadjoint nilpotent orbits under the action of $\Gamma$. For general results on nilpotent orbits over algebraically closed fields see~\cite{Carter,CM}.

\begin{defn}
We say that a nilpotent orbit under $\Gamma$ in $\fg^*$ is \emph{principal} if it is Zariski dense in the nilpotent cone $\mathcal{N}(\fg^*)$.
We say that $\varphi\in \fg^*$ is a 
 \emph{principal nilpotent element} if its orbit is principal. 
 
 We say that a nilpotent $\varphi\in \fg^*$ is \emph{principal in a Levi} (or PL for short) if there exists a $\K$-Levi subalgebra $\fl\subset \fg$ and a nilpotent element $f\in \fl$ such that the Killing form pairing with $f$ defines $\varphi$ in $\fg^*$, and a principal nilpotent element of $\fl^*$. We call a nilpotent $\Gamma$-orbit in $\fg^*$ a \emph{PL-orbit} if it consists of PL elements.
\end{defn}

We remark that if $G$ is quasi-split then a nilpotent element $\varphi\in \fg^*$ is principal if and only if it is regular, {\it i.e.} the dimension of its centralizer equals the rank of $\fg$. 
\begin{lemma}\label{lem:nN}
Let $\fn$ be the nilpotent radical of the Lie algebra of a minimal parabolic subgroup $P_0$. Then $\fn$ intersects any nilpotent orbit under $\Gamma$ in $\fg$.
\end{lemma}
\begin{proof}
Let $f\in \fg$ be nilpotent, and complete to an $\sll_2$-triple $(e,h,f)$. Then $h$ defines a parabolic subgroup $P$ which then includes a minimal parabolic $Q_0$. Then $Q_0$ is conjugate to $P_0$ under $\Gamma$ (see \cite[Thm.~4.13(b)]{BT}). Since $f$ lies in the nilpotent radical of the Lie algebra of $P$, its conjugate will lie in $\fn$.
\end{proof}

\begin{defn}
We say that a Whittaker pair $(S,\varphi)$ is \emph{standard} if $\fn_{S,\varphi}$ is the nilpotent radical of the Lie algebra of a minimal parabolic subgroup of $G$.
In this case we will call the Fourier coefficient $\cF_{S,\varphi}$ a \emph{Whittaker coefficient}. 
\end{defn}

\begin{remark}
In \cite[\S 6]{GGS} the  Whittaker coefficients are called principal degenerate Whittaker--Fourier coefficients.
\end{remark}

\begin{cor}\label{cor:PrinNeut}
A nilpotent element $\varphi\in \fg^*$ is principal if and only if it can be completed to a neutral standard Whittaker pair. 
\end{cor}
\begin{proof}
Let $h$ complete $\varphi$ to a neutral standard Whittaker pair. Then $\fn_{h,\varphi}=\fg^h_{>1}=\fg^h_{>0}$ is the nilpotent radical of the Lie algebra of a minimal parabolic subgroup, and thus so is \(\fn:=\fg^h_{<1}\).   Let $f\in \fn$  define $\varphi$ through the Killing form pairing. Then we have $[f,\fg^h_{\leq 0}]=\fn$ and thus $\Gamma f\cap \fn$ is Zariski open, and thus Zariski dense, in $\fn$. The statement follows now from Lemma \ref{lem:nN}.

Conversely, let $\varphi\in \fg^*$ be a principal nilpotent, and let $(e,h,f)$ be an $\sll_2$-triple such that $f$ defines $\varphi$ via the Killing form. Let $\cO$ denote  the complex orbit of $f$ and $\bar \cO$ denote its Zariski closure. Then $\bar \cO= \cN(\fg)\supset \fg^h_{<0}$. Thus $\cO$ is the Richardson orbit for $\fg^h_{\leq 0}$, and thus $\dim \cO=2\dim \fg^h_{<0}$. Now suppose by way of contradiction that the pair $(h,\varphi)$\ is not standard. Then $\fg^h_{\leq 0}$ is not a minimal $\K$-parabolic subalgebra, {\it i.e.} there exists a smaller $\K$-parabolic subalgebra $\mathfrak{p}$ with nilpotent radical $\fn\supsetneq \fg^h_{<0}$. But  $\fn \subset \cN=  \bar \cO$, and thus $\cO$ is a Richardson orbit for $\mathfrak{p}$, thus $\dim \cO = 2\dim \fn>2\dim \fg^h_{<0}=\dim \cO$ - contradiction.
\end{proof}

\begin{cor}
A nilpotent $\varphi\in \fg^*$ is PL if and only if it can be completed to a standard Whittaker pair $(S,\varphi)$. 
\end{cor}
\begin{proof}
Let $(S,\varphi)$ be a standard Whittaker pair. Then $S=h+Z$ where $(h,\varphi)$ is neutral and commutes with $Z$.
Then $Z$ defines a Levi subalgebra $\fl$, and the Whittaker pair $(h,\varphi)$ is neutral and standard in $\fl$. By Corollary \ref{cor:PrinNeut}, $\varphi$ is principal in $\fl$.

Conversely, if $\varphi$ is principal in $\fl$ and $Z$ defines $\fl$ we let $S:=TZ+h$ for $T\in \Q_{>0}$ big enough. Then  $(S,\varphi)$ is a standard Whittaker pair.
\end{proof}
Let us remark that in \cite{GGS} a different definition of principal and PL elements was given. The following lemma states the equivalence of the definitions.
\begin{lemma}
Let $\varphi\in \fg^*$ be nilpotent. Then 

\begin{enumerate}
 \item $\varphi$ is PL if and only if there exist a maximal split toral subalgebra $\fa$ of $\lie g$ and a choice of associated simple roots $\Pi$ such that $\varphi \in \bigoplus_{\alpha_i \in \Pi} \lie g^*_{\alpha_i}$, where $g^*_{\alpha_i}$ denotes the dual of the root space $\fg_{\alp_i}$.
\item
If $\varphi \in \bigoplus_{\alpha_i \in \Pi} \lie g^*_{\alpha_i}$ then $\varphi$ is principal in the Levi {subalgebra} given by those simple roots $\alp_i$ for which the projection of $\varphi$ to $\lie g^*_{\alpha_i}$ is non-zero.
\end{enumerate}
\end{lemma}
\begin{proof}
Let $\sum_{\alp_i\in \Pi} \fg^{\times}_{\alp_i}\subset \bigoplus_{\alpha_i \in \Pi} \lie g^*_{\alpha_i}$ denote the subset of vectors with all projections non-zero. It is enough to show that $\varphi\in \fg^*$ is principal if and only if there exist $(\fa,\Pi)$ as above such that $\varphi \in \sum_{\alp_i\in \Pi} \fg^{\times}_{\alp_i}$.

To show that, assume first that $\varphi$ is principal. Then, by Corollary \ref{cor:PrinNeut}, $\varphi$ can be completed to a neutral standard pair $(h,\varphi)$. Then $h$ defines a torus and simple roots, and we have $\varphi\in \sum_{\alp_i\in \Pi} \fg^{\times}_{\alp_i}$. 
Conversely, give $\fa$ and $\Pi$ as above, we let $h:=\sum c_i \alp_i^\vee$, where $\alp_i^\vee$ are the coroots given by scalar product with $\alp_i$, and $c_i$ are chosen such that $\varphi\in (\fg^*)^h_{-2}$. Then $(h,\varphi)$ is a standard Whittaker pair. Moreover, $\varphi$ is a generic element of
$(\fg^*)^h_{-2}$ and thus the Jacobson--Morozov theory implies that $(h,\varphi)$ is a neutral pair.
\end{proof}

\begin{remark}
Note that for $G=\GL_n(\A)$ all orbits $\mathcal{O}$ are PL-orbits. In general this is, however, not the case, see Appendix~\ref{subsec:PL} for details.
\end{remark}

\pagebreak
\subsection{Levi-distinguished Fourier coefficients}\label{subsec:Levi-DistPrel}
\begin{definition}
We say that a nilpotent $f \in \fg$ is \emph{$\K$-distinguished}, if it does not belong to a proper $\K$-Levi subalgebra $\fl \subsetneq \fg$. 
In this case we will also say that $\varphi\in \fg^*$ given by the Killing form pairing with $f$ is $\K$-distinguished. We will also say that the orbit of $\varphi$ is $\K$-distinguished.
\end{definition}

\begin{example}
The nilpotent orbits for $\mathrm{Sp}_{2n}(\C)$ are given by partitions of $2n$ such that odd parts have even multiplicity. Each such orbit, except the zero one, decomposes to infinitely many $\mathrm{Sp}_{2n}(\Q)$-orbits - one for each collection of equivalence classes of quadratic forms $Q_1,\dots,Q_k$ of dimensions $m_1,\dots m_k$ where $k$ is the number of even parts in the partition and $m_1,\dots m_k$ are the multiplicities of these parts. 
A complex orbit is distinguished over $\mathbb{C}$ (\ie does not intersect a proper Levi subalgebra) if and only if all parts have multiplicity one (and thus there are no odd parts). To see the ``only if'' part note that  if the partition includes a part $k$ with multiplicity two then the orbit intersects the Levi $\GL_k \times   \mathrm{Sp}_{2(n-k)}$. If $k$ is odd then this Levi is defined over $\Q$ and thus all $\Q$-distinguished orbits correspond to totally even partitions. If $k$ is even then this Levi is defined over $\Q$ if and only if the quadratic form on the multiplicity space of $k$ is anisotropic. Thus, we obtain that a necessary condition for an orbit $\cO$ to be $\Q$-distinguished is that its partition $\lam(\cO)$ is totally even, a sufficient condition is that $\lam(\cO)$ is multiplicity free, and for totally even partitions with multiplicities there are infinitely many $\Q$-distinguished orbits and at least one not $\Q$-distinguished.  For example, for the partition $(4,2)$ all orbits in $\mathfrak{sp}_{6}(\Q)$ are $\Q$-distinguished, for the partition $2^3$ some orbits are $\Q$-distinguished and some are not, and all other partitions do not correspond to $\Q$-distinguished
orbits.
\end{example}

\begin{lem}\label{lem:PrinDist}
Every principal nilpotent element is $\K$-distinguished. 
\end{lem}
\begin{proof}
Let $f\in \fg$ define a principal nilpotent element via the Killing form. Suppose the contrary: $f$ lies in a proper $\K$-Levi subalgebra $\fl$ of $\fg$. Let $Z\in \fg$ be a rational semi-simple element that defines $\fl$. Complete $f$ to an $\sll_2$-triple $\gamma:=(e,h,f)$ in $\fl$. Then $\ad(Z)$ acts by a scalar on every irreducible submodule of the adjoint action of $\gamma$ on $\fg$. Since  $\fl\neq \fg$, there exists an irreducible submodule $V$ on which $\ad(Z)$ acts by a negative scalar $-c$. Let $v$ be a highest weight vector of $V$ of weight $d$, and let $S:=h+c^{-1}(d+2)Z$. Then $v+f\in \fg^S_{-2}$ and thus $v+f$ is nilpotent. Since $f$ is principal, $v+f$ lies in the Zariski closure of $\Gamma f$. 
On the other hand, $v+f$ belongs to the affine space $f+\fg^e$, which is called the Slodowy slice to $\Gamma f$ at $f$, and is transversal to $\Gamma f$, contradicting the assumption that $v+f$ lies in the Zariski closure of $\Gamma f$.
\end{proof}

\begin{lemma}\label{lem:LeviConj}
Let $f\in \fg$ be nilpotent. Then all $\K$-Levi subalgebras $\fl\subseteq \fg$ such that $f \in \fl$ and $f$ is $\K$-distinguished in $\fl$ are conjugate by the centralizer of $f$.
\end{lemma}
\begin{proof}
Complete $f$ to an $\sl_2$-triple $\gamma :=(e,h,f)$ and denote its centralizer by $G_{\gamma}$. Let us show that all  $\K$-Levi subalgebras $\fl$ of $\fg$ that contain $\gamma$ and in which $f$ is distinguished are conjugate by $G_\gamma$. Let $\fl$ be such a subalgebra, $L\subseteq G$ be the corresponding Levi subgroup, and let $C$ denote the maximal split torus of the center of $L$. Then $C$ is a split torus in $G_{\gamma}$. Let us show that it is a maximal split torus. Let $T\supseteq C$ be a larger split torus in $G_{\gamma}$. Then, the centralizer of $T$ in $\fg$ is a $\K$-Levi subalgebra that lies in $\fl$ and includes $\gamma$, and thus is equal to $\fl$. Thus  $T=C$. 

Since $\fl$ is the centralizer of $T$ in $G$, $T$ is a maximal split torus of $G_{\gamma}$, and all maximal split tori of reductive groups are conjugate (see \cite[15.14]{Bor}),  we get that all the choices of $L$ are conjugate.

Since all the choices of $\gamma$ are conjugate by the centralizer of $f$, the lemma follows.
\end{proof}

\begin{definition}
\label{def:Levi-distinguished}
Let $Z\in \fg$ be a rational-semisimple element and $\fl$ denote its centralizer. Let $(h,\varphi)$ be a neutral Whittaker pair for $\fl$, such that the orbit of $\varphi$ in $\fl^*$ is $\K$-distinguished. 
We say that the Whittaker pair $(h+Z,\varphi)$ is 
 \emph{Levi-distinguished} if\begin{equation}\label{=LeviDist0}
\fg^{h+Z}_{> 1}=\fg^{h+Z}_{\geq 2}=
\fg^{Z}_{>0}\oplus \fl^{h}_{\geq 2}\text{ and }\fg^{h+Z}_{ 1}=\fl^{h}_{1}.
\end{equation}
In this case we also say that the 
Fourier coefficient $\cF_{h+Z,\varphi}$ is  \emph{Levi-distinguished}.
\end{definition}

\begin{remark}
Let $(h,\varphi)$ be a neutral Whittaker pair for $\fg$. 
If $\varphi$ is $\K$-distinguished then $\cF_{h,\varphi}$ is a Levi-distinguished Fourier coefficient. If a rational semi-simple $Z$ commutes with $h$ and with $\varphi$, and $\varphi$ is $\K$-distinguished in $\fl:=\fg^Z_0$ then $\cF_{h+TZ,\varphi}$ is a Levi-distinguished Fourier coefficient for any $T$ bigger than $m/M+1$, where $m$ is the maximal eigenvalue of $h$ and $M$ is the minimal positive eigenvalue of $Z$. See also Lemma~\ref{lem:LeviDist} for further discussion. 
\end{remark}

\begin{lem}[{\cite[Lemma 3.0.2]{GGS}}]\label{lem:Z}
For any Whittaker pair $(H,\varphi)$ there exists $Z\in \fg^{H}_0$ such that $(H-Z,\varphi)$  is a neutral Whittaker pair.
\end{lem}
\begin{remark}
In \cite{GGS} the lemma is proven over a local field, but the proof only used the Jacobson--Morozov theorem, that holds over arbitrary fields of characteristic zero.
\end{remark}

\begin{lemma}\label{lem:WhitPL}
For any Whittaker pair $(H,\varphi)$, the following are equivalent:
\begin{enumerate}
\item $(H,\varphi)$ is standard pair
\item $(H,\varphi)$ is a Levi-distinguished Fourier coefficient, and $\varphi$ is a PL nilpotent.
\end{enumerate}
\end{lemma}
\begin{proof}
First let $(H,\varphi)$ be a standard pair. Then by Lemma \ref{lem:Z}, $H$ can be decomposed as $H=h+Z$ where $(h,\varphi)$ is a neutral pair and $Z$ commutes with $h$ and with $\varphi$. Let $\fl$ and $L$ denote the centralizers of $Z$ in $\fg$ and $G$, and  $N:=N_{H,\varphi}$. Then $N$ is the unipotent radical of a minimal parabolic subgroup of $G$, and $L$ is a Levi subgroup of $G$. Thus, $N\cap L$ is the unipotent radical of a minimal parabolic subgroup of $L$. The Lie algebra of $N\cap L$ is $\fn_{H,\varphi}\cap \fg^Z_0=\fg^{h}_{\geq 1}\cap \fg^Z_0$. Thus, $\Exp(\fg^h_{\leq -1}\cap \fg^Z_0)$ is the unipotent radical of a minimal parabolic subgroup of $L$. Since $\varphi$ is given by Killing form pairing with $f\in \fg^h_{\leq -1}\cap \fg^Z_0$, we get by Corollary \ref{cor:PrinNeut} that $\varphi$ is principal in $\fl$.  Replacing $Z$ by $tZ$ with $t$ large enough, we obtain that $(H,\varphi)$ is a Levi-distinguished pair.

Now, assume that $\varphi$ is a PL nilpotent, and let  $\cF_{h+Z,\varphi}$ be a Levi-distinguished pair. Let $\fl=\fg^Z_0$ be the corresponding Levi, and let $f=f_{\varphi}$ be the element of $\fg$ that defines $\varphi$. Since $f$ is distinguished in $\fl$, and principal in some Levi, Lemmas \ref{lem:PrinDist} and \ref{lem:LeviConj} imply that $f$ is principal in $\fl$. Thus, $\fn_{H,\varphi}\cap \fl$ is the nilpotent radical of the Lie algebra of a minimal parabolic subgroup of $L$ and thus $\fn_{H,\varphi}=\fn_{H,\varphi}\cap \fl \oplus \fg^Z_{>0}$ the nilpotent radical of the Lie algebra of a minimal parabolic subgroup of $G$. Thus $\cF_{H,\varphi}$ is standard Whittaker pair.
\end{proof}

\begin{lemma}\label{lem:SameDim}
 Let $(H,\varphi)$ and $(S,\varphi)$ be Levi-distinguished Whittaker pairs with the same $\varphi$.  Then $\dim \fn_{S,\varphi}=\dim \fn_{H,\varphi}$
\end{lemma}
\begin{proof}
By the definition of Levi-distinguished Whittaker pair, there exists a decomposition $H=h+Z$ 
such that $\varphi$ is given by the Killing form pairing with a distinguished element $f\in \fl=\fg^Z_0$ and \eqref{=LeviDist0} holds. Let $S=h'+Z'$ be a decomposition for $S$ satisfying the same properties, and let $\fl':=\fg^{Z'}_0$. By Lemma \ref{lem:LeviConj}, $\fl$ and $\fl'$ are conjugate by the centralizer of $f$, and thus we can assume that $\fl=\fl'$. By the Jacobson--Morozov theorem,  $h$ and $h'$ are conjugate by the centralizer of $f$ in $\fl$. Now,
\begin{equation}
    \dim \fn_{H,\varphi}=\dim \fl^h_{\geq 2}+\dim \fg^Z_{>0}=\dim\fl^h_{\geq 2}+(\dim \fg-\dim \fl)/2= \dim \fn_{S,\varphi}\,.
    \qedhere
\end{equation}
\end{proof}
In Lemma \ref{lem:dim} \ref{it:DimLeviDist} below we show that Levi-distinguished pairs have maximal dimension of $\fn_{S,\varphi}$ among all Whittaker pairs with nilpotent element $\varphi$.

\subsection{Order on nilpotent orbits and Whittaker support}
 
\begin{definition}\label{def:order}
    We define a partial order on nilpotent orbits in $\fg^* = \fg^*(\K)$ to be the transitive closure of the following relation $R$: $(\cO,\cO')\in R$ if $\cO\neq \cO'$ and 
there exist $\varphi\in \cO,$
 rational semi-simple $H,Z\in \fg$, and $\varphi'\in (\fg^*)^Z_{>0 } \cap (\fg^*)^H_{-2}$ such that $\varphi\in (\fg^*)^Z_0 \cap (\fg^*)^H_{-2 }$, $[H,Z]=0,$ and 
    $\varphi+\varphi'\in  \cO'$.
\end{definition}
In the notation of the introduction, the conditions in the definition read $\varphi+\varphi'\cge_H \varphi$, the parabolic subalgebra in \eqref{=IntOrder} being $\fg^Z_{\geq 0}$. 
In Appendix~\ref{sec:Geo}, we study these rational orbits in more detail. In particular, in Corollary~\ref{cor:order} we prove that this is indeed a partial order, {\it i.e.} that $R$ is anti-symmetric.
We will thus denote $\cO\leq \cO'$ (or $\cO'\geq \cO$) if $(\cO,\cO')$ lies in the transitive closure of $R$, and $\cO<\cO'$ if $\cO\leq \cO'$ and $\cO\neq \cO'$. By Corollary \ref{cor:complexDim} below this implies an inequality on the dimensions of complexifications: $\dim \cO_{\C}<\dim \cO'_{\C}$.

\begin{lemma}
\label{lem:orbit-closure}
If $\cO'$ is bigger than $\cO$, {\it i.e.} if $(\cO,\cO')\in R$, then for any place $\nu$ of $\K$, the closure of $\cO'$ in $\fg(\K_{\nu})$ (in the local topology) contains $\cO$.
\end{lemma}
\begin{proof}
It is enough to show that for any $Z\in \fg$, $\varphi\in \fg^Z_0$ and $\psi \in \fg^Z_{>0}$, $\varphi$ lies in the closure of ${\bf G}(\K_{\nu})(\varphi+\psi)$. Let $\eps_i\in \K_{\nu}$ be a sequence converging to zero and let $g_{i}:=\exp(-\eps_i Z)$. Then $g_i$ centralize $\varphi$, while $g_i\psi\to 0$. Thus $g_i(\varphi+\psi)\to \varphi$.
\end{proof}

\begin{remark}
For $\bf G= \GL_n,$ our order coincides with the closure order on complex orbits (see \cite[Proposition 7.0.5]{GGS:support}).
However, in general this is not the case. For example, one can show that the maximal elements with respect to this order are the $\K$-distinguished orbits.
\end{remark}

\begin{defn}\label{def:WO}
  For an automorphic function $\eta$, we define $\mathrm{WO}(\eta)$ to be the set of nilpotent orbits $\cO$ in $\fg^*$ under the coadjoint action of $\G(\K)$ such that $\cF_{h,\varphi}[\eta]\neq 0$ for some neutral Whittaker pair $(h,\varphi)$ with $\varphi\in \cO$. Using the partial order of Definition~\ref{def:order}, we define the \emph{Whittaker support} $\WS(\eta)$ to be the set of maximal elements in $\mathrm{WO}(\eta)$.
\end{defn}

\begin{remark}
The definition of $\WS$ of automorphic representations given in \cite[\S 8]{GGS:support} uses the closure (Bruhat) order on complex orbits, which is coarser then the one we use here. Thus, the $\WS$ defined in Definition \ref{def:WO} includes the $\WS$ used in \cite[\S 8]{GGS:support} but is frequently not equal to it.
\end{remark}

\section{Relating different Fourier coefficients}\label{sec:GenLemmas}

In this section we will introduce some standard tools used to relate different types of Fourier coefficients.
In the subsequent subsection we first relate Fourier coefficients for different isotropic subspaces for a single Whittaker pair.
Then we relate Fourier coefficients along deformations of Whittaker pairs, and lastly we explain the relationship between the conjugation of a Whittaker pair and the translation of the argument of a Fourier coefficient.

The statements in this section have partial local analogues in \cite{GGS,GGS:support}. However, the statements we give here are global and more explicit.

\subsection{Relating different isotropic subspaces}

We will now see how $\cF_{S,\varphi,\varphi'}$ and $\cF_{S,\varphi,\varphi'}^R$ can be expressed through each other.
\begin{lemma}[cf. {\cite[Lemma 6.0.2]{GGS}} for a slightly weaker statement]\label{lem:StvN}
  Let $\eta \in C^{\infty}(\Gamma \backslash G)$, let $(S,\varphi,\varphi')$ be a Whittaker triple, $\fn_{S,\varphi}$ be as in \eqref{=Nsphi}, and $\fu := \mathfrak{g}^S_{\geq 1}$.
  
  Let $\fn_{S,\varphi} \subseteq \fri \subseteq \fr $ be isotropic subspaces of $\fu$, and let $\fri^\perp\supseteq \fr^{\perp}$ be their orthogonal complements with respect to $\omega_\varphi|_\fu$.
  Let also $I: = \Exp(\fri(\mathbb{A}))$, $R = \Exp(\fr(\mathbb{A}))$, $I^\bot:=\Exp(\fri^\bot(\mathbb{A}))$ and $R^\bot:=\Exp(\fr^\bot(\mathbb{A}))$.
  Then,
\begin{align}\label{eq:FL-as-F}
    \cF_{S,\varphi,\varphi'}^R[\eta](g) &= \intl_{[R/I]} \cF^I_{S,\varphi,\varphi'}[\eta](ug) \, du \, \\
    \intertext{and}
    \label{eq:F-as-FL}
    \cF^I_{S,\varphi,\varphi'}[\eta](g) &=\sum_{\gamma \in \Exp(\fri^{\perp}/\fr^\bot)}\cF_{S,\varphi,\varphi'}^{R}[\eta](\gamma g).
\end{align}
\end{lemma}

We will mostly use this lemma in the case $\fri=\fn_{S,\varphi}$ for which $\fri^{\bot}=\fu$.

\begin{proof}
We assume that $\varphi$ is non-zero since otherwise $R=I=N_{S,\varphi}$. We have that $I \subseteq R$ with $R/I$ abelian which means that \eqref{eq:FL-as-F} follows immediately from the definitions of $\cF^I_{S,\varphi,\varphi'}$ and $\cF_{S,\varphi,\varphi'}^R$. 
For \eqref{eq:F-as-FL} observe that the function 
$(\chi_{\varphi}^{R})^{-1}\cdot \cF^I_{S,\varphi,\varphi'}[\eta]$ on $R$ is left-invariant under the action of $I\cdot(R\cap \Gamma)$. In other words, we can identify it with a function on
\begin{equation}\label{=LN}
(I\cdot(R\cap \Gamma)) \backslash R   \cong \Exp(\fr/\fri\bigr) \backslash \Exp((\fr/\fri)(\A)) =:[R/I],
\end{equation}
where the  equality follows from the fact that $R/I$ is abelian. Therefore, we have a Fourier series expansion
\begin{equation}
 \cF^I_{S,\varphi,\varphi'}[\eta](u)=\sum_{\psi\in [R/I]^{\wedge}}c_{\psi,\chi_{\varphi+\varphi'}^{R}}(\eta)\psi(u)\chi_{\varphi}^{R}(u),
\end{equation}
where $[R/I]^{\wedge}$ denotes the Pontryagin dual group of $[R/I]$ and 
\begin{equation}\label{eq:c_Whittaker-Fourier_coefficient}
 c_{\psi,\chi_{\varphi}^{R}}(\eta)=\intl_{[R]}\psi(u)^{-1}\chi_{\varphi+\varphi'}^{R}(u)^{-1}\eta(u)du.
 \end{equation}
In particular, denoting by $\Id\in G$ the identity element we obtain
\begin{equation}\label{=c_e}
 \cF^I_{S,\varphi,\varphi'}[\eta](\Id)=\sum_{\psi\in [R/I]^{\wedge}}c_{\psi,\chi_{\varphi+\varphi'}^{R}}(\eta).
\end{equation}

Now observe that the map $X \mapsto \omega_{\varphi}(X,\cdot)=\varphi\circ \ad(X)$ induces an isomorphism between $\fri^{\bot}/\fr^\bot$ and the dual space $(\fr/\fri)^*$. Hence, according to equations (\ref{eq:chi_isomorphism}) and \eqref{=LN}, we can use the character $\chi$ to define a group isomorphism
\begin{equation}\label{=DualL}
 \begin{array}{rcl}
  (I^{\bot}\cap \Gamma)/(R^\bot\cap \Gamma) & \longrightarrow & [R/I]^{\wedge} \\
      u & \mapsto & \psi_{u},
 \end{array}
\end{equation}
where
\begin{align}
 \psi_{u}(r)=\chi(\varphi([X,Y])), \qquad  \mbox{$u=\exp X$}\quad \mbox{and}\quad \mbox{$r=\exp Y$.}
\end{align}
Hence, for all $u\in I^{\bot}\cap \Gamma$ and $r\in R$ we have
\begin{eqnarray*}
 \psi_{u}(r)\chi_{\varphi+\varphi'}^{R}(r)  =  \chi(\varphi([X,Y])+\varphi'([X,Y]))\chi(\varphi(Y)+\varphi'(Y))                   =  \chi((\varphi+\varphi')(Y+[X,Y]))\\
                   =  \chi((\varphi+\varphi')(e^{\ad(X)}(Y)))
                   =  \chi_{\varphi+\varphi'}((\Ad(u)Y))
                  =  \chi_{\varphi+\varphi'}^{R}(uru^{-1}).
\end{eqnarray*}
Here we are taking again $u=\exp X$, $r=\exp Y$ and the middle equality follows from the vanishing of $\varphi$  on $\fg^{S}_{>2}$. But now, from formula \eqref{eq:c_Whittaker-Fourier_coefficient} and the fact that $\eta$ is automorphic, we have
\begin{eqnarray}
  c_{\psi_{u},\chi_{\varphi+\varphi'}^{R}}(\eta) & = & \intl_{[R]}\psi_{u}(r)^{-1}\chi_{\varphi+\varphi'}^{R}(r)^{-1}\eta(r)dr
              =  \intl_{[R]}\chi_{\varphi+\varphi'}^{R}(uru^{-1})^{-1}\eta(r)dr.\nonumber\\
             & = & \intl_{[R]}\chi_{\varphi+\varphi'}^{R}(r)^{-1}\eta(u^{-1}ru)dr               =  \cF_{S,\varphi,\varphi'}^{R}[\eta](u),
\end{eqnarray}
for all $u\in U\cap \Gamma$.  Combining this with \eqref{=c_e} and \eqref{=DualL} we obtain
\begin{equation}
 \cF^I_{S,\varphi,\varphi'}[\eta](\Id)=\sum_{u\in(I^{\bot}\cap \Gamma)/(R^\bot\cap \Gamma)}\cF_{S,\varphi,\varphi'}^{R}[\eta](u).
\end{equation}
Applying this to $\eta$ and its right shifts we obtain \eqref{eq:F-as-FL}.
\end{proof}

\begin{cor}\label{cor:RootExchange}
Let $\eta \in C^{\infty}(\Gamma \backslash G)$, let $(S,\varphi,\varphi')$ be a Whittaker triple, and $\fn_{S,\varphi}$, and $\fu$ be as above.
Let $\fr,\fr'\subseteq \fu$ be two isotropic subspaces that include $\lie n_{S,\varphi}$. Assume $\dim \fr =\dim \fr'$ and $\fr\cap (\fr')^{\bot}\subseteq \fr'$. Then 
\begin{equation}
\cF_{S,\varphi,\varphi'}^R[\eta](g) = \intl_{R/(R\cap R')} \cF^{R'}_{S,\varphi,\varphi'}[\eta](ug) \, du\,.
\end{equation}
Note that this is a non-compact, adelic, integral.
\end{cor}
\begin{proof}
Let $\fri:=\fr\cap \fr'$. 
We claim that the natural map $p:\fr/\fri \to \fri^{\bot}/(\fr')^{\bot}$ is an isomorphism.
Indeed, from the assumption $\fr\cap (\fr')^{\bot}\subseteq \fr'$ we have $\fri = \fr\cap (\fr')^{\bot}$  and thus $p$ is an embedding. Further, from the assumption $\dim \fr=\dim \fr'$ we obtain that the source and the target spaces of $p$ have the same dimension. Indeed,
\begin{equation}
\dim (\fri^{\bot}/(\fr')^{\bot})= \dim \fu -\dim \fri - (\dim \fu -\dim \fr')=\dim \fr-\dim \fri= \dim (\fr/\fri)
\end{equation}
Now, $p$ defines a natural isomorphism $\Exp(\fri^\perp/(\fr')^{\bot})\IsoTo{}\Exp(\fr/\fri)$.
Let $I:=\Exp(\fri(\A))=R\cap R'$.
From Lemma \ref{lem:StvN} we obtain
\begin{equation}
\cF_{S,\varphi,\varphi'}^R[\eta](g) = 
 \intl_{[R/I]} \sum_{\gamma \in \Exp(\fr/\fri)} \cF^{R'}_{S,\varphi,\varphi'}[\eta](\gamma ug) \, du
=
\intl_{R/I} \cF^{R'}_{S,\varphi,\varphi'}[\eta](ug) \, du\,.  \qedhere
\end{equation}
\end{proof}
This corollary can be seen as a version of the root exchange lemma in \cite{GRS}. 

\subsection{Relating different Whittaker pairs}\label{subsec:rel}

Let $(H,\varphi)$ be a Whittaker pair.

\begin{lem}
Let $Z$ be as in Lemma \ref{lem:Z}. Then $(H-Z,\varphi)$ dominates $(H,\varphi)$.
\end{lem}
\begin{proof}
Denote $h:=H-Z$. We have to show that \eqref{=domin} holds, {\it i.e.}
\begin{equation}
\fg_{\varphi}\cap \fg^h_{\geq 1}\subseteq \fg^Z_{\geq0}\,.
\end{equation}
Since $\fg_{\varphi}$ is spanned by lowest weight vectors, we have $\fg_{\varphi}\subseteq \fg^h_{\leq 0}$ and thus $\fg_{\varphi}\cap \fg^h_{\geq 1}=\{0\}$.
\end{proof}

\begin{cor}
\label{cor:domin-neutral}
Any Whittaker pair is dominated by a neutral Whittaker pair with the same character $\varphi$.
\end{cor}

Another example of domination is provided by the following proposition, that immedi\-ately follows from \cite[Proposition 3.3.3]{GGS}.

\begin{prop}\label{prop:PL}
If $\varphi$ is a PL nilpotent then there exists $Z\in \fg$ such that $(H+Z,\varphi)$ is a standard Whittaker pair and $(H,\varphi)$ dominates $(H+Z,\varphi)$. 
\end{prop}
From now till the end of the section let $Z\in \fg_0^H$ be a rational semi-simple element such that $(H,\varphi)$ dominates $(H+Z,\varphi)$.
We will now consider the deformation of the former Whittaker pair to the latter.
For any rational number $t \geq 0$ define \begin{equation}\label{=ut}
H_t:=H+tZ,\quad \fu_t:=\fg^{H_t}_{\geq 1},\quad \fv_t:=\fg^{H_t}_{> 1},\text{ and }\fw_t:=\fg^{H_t}_{1}. \quad
\end{equation}
\begin{defn}\label{def:crit}
We call $t \geq 0$ \emph{regular} if $\fu_t = \fu_{t+\eps}$ for any small enough $\eps\in \Q$, or in other words $\fw_t\subset \fg^Z_0$. If $t$ is not regular we call it \emph{critical}. Equivalently, $t$ is critical if $\fg^{H_t}_{1}\nsubseteq \fg^{Z}_{0}$ which we may interpret as something new has entered the $1$-eigenspace of $H$.
For convenience, we will say that $t=0$ is critical.

We also say that $t \geq 0$ is \emph{quasi-critical} if either $\fg^{H_t}_{1}\nsubseteq \fg^{Z}_{0}$ or $\fg^{H_t}_{2}\nsubseteq \fg^{Z}_{0}$. We may interpret this as something new has entered either the $1$-eigenspace or the $2$-eigenspace. The latter is related to new characters being available in the Whittaker pairs.
Note that there are only finitely many critical numbers.
\end{defn}

Recall the anti-symmetric form $\omega_\varphi$ on $\fg$ given by $\omega_\varphi(X,Y)=\varphi([X,Y])$ and the definition $\lie n_{H_t,\varphi} := \Ker(\omega_\varphi|_{\fu_t})$.
\begin{lemma}[{\cite[Lemma 3.2.6]{GGS}}]\label{lem:help}\hfill
\begin{enumerate}
\item \label{it:OmInv} The form $\omega_\varphi$ is $\ad(Z)$-invariant.
\item \label{it:KerOm} $\Ker \omega_\varphi = \fg_\varphi$.
\item \label{it:KerNeg} $\Ker(\omega_\varphi|_{\fw_t})=\Ker(\omega_\varphi)\cap \fw_t$. 
\item \label{it:Kerv} $\Ker(\omega_\varphi|_{\fu_t})=\fv_t\oplus \Ker(\omega_\varphi|_{\fw_t}) $.
\item \label{it:LW}  $\fw_s\cap\fg_\varphi\subseteq \fu_{t}$  for any $s<t$.
\end{enumerate}
\end{lemma}
We will often suppress the deformation $H_t$ and character $\varphi$ and simply write $\fn_t = \fn_{H_t,\varphi}$. Similarly, define
\begin{equation}\label{=lt}
  \fl_t = \fl_{H_t,\varphi} := (\fu_t\cap \fg^{Z}_{< 0})+\fn_{H_t,\varphi} \quad \text{ and } \quad \fr_t = \fr_{H_t,\varphi}:=(\fu_t\cap \fg^{Z}_{> 0})+\fn_{H_t,\varphi}.
\end{equation}
We note that $\fl_t$ and $\fr_t$ are nilpotent subalgebras. The choice of notation for them comes from `left' and `right'.

\begin{lemma}
\label{lem:key}
    For any $t\geq0$ we have
\begin{enumerate}
\item \label{it:lrCom} $\fl_t$ and $\fr_t$ are ideals in $\fu_t$ and $[\fl_t,\fr_t]\subseteq \fl_t\cap\fr_t=\fn_t$.

\item \label{it:MaxIs} $\fl_t$ and $\fr_t$  are isotropic subspaces of $\fu_t$ of the same dimension, and the natural projections $\fl_t/\fn_t\to \fu_t/\fr_t^\bot$ and $\fr_t/\fn_{t} \to \fu_t/\fl_t^\bot$ are isomorphisms.
Furthermore, $\lie l_t = \lie g^{H_t}_1 \cap \lie g^Z_{<0} \oplus \lie n_t$.
    
\item \label{it:Emb} Suppose that $0\leq s< t$, and all the elements of the interval $(s,t)$ are regular.
Then 
\begin{equation}\label{=vw}
\fv_{t}\oplus (\fw_{t}\cap \fg^Z_{<0})=\fv_{s}\oplus (\fw_{s}\cap \fg^Z_{>0})
\end{equation}
\begin{equation}\label{=lrw}
\fl_{t}=\fr_{s}+ (\fw_{t}\cap \fg_{\varphi}) \text{ and }\fr_{s}\cap (\fw_{t}\cap \fg_{\varphi})=\fw_{0}\cap \fg^Z_0\cap \fg_{\varphi}.
\end{equation}
Moreover, $\fr_{s}$ is an ideal in $\fl_{t}$ and the quotient is commutative.
\end{enumerate}
\end{lemma}
\begin{proof}

  It is easy to see that $\fv_t$ is an ideal in $\fu_t$ with commutative quotient, and that $\fv_t\subseteq \fl_t\cap\fr_t=\fn_t$. This proves \ref{it:lrCom}. For the first part of \ref{it:MaxIs}, note that $(\fl_t+\fr_t)/\fn_t$ is a symplectic space in  which the projections of $\fl_t$ and $\fr_t$ are complementary Lagrangians. 

For the second part, we have by Lemma~\ref{lem:help} that $\lie g^{H_t}_1 \cap \lie g_\varphi \subseteq \lie g^Z_{\geq 0}$ and thus,
\begin{equation}
    \label{eq:lt}
    \fl_t=\fv_{t}\oplus (\fw_{t}\cap \fg^Z_{<0})\oplus (\fw_{t}\cap \fg_{\varphi}) \, .
\end{equation}

For part \ref{it:Emb} note that 
\begin{align}
\fv_{s}&=(\fv_{s}\cap \fg^Z_{\geq0}) \oplus (\fv_{s}\cap \fg^Z_{<0}) \,,\\
\fv_{t}&= (\fv_{t}\cap \fg^Z_{<0}) \oplus (\fv_{t}\cap \fg^Z_{\geq0}) \,,\\
\fv_{t}\cap \fg^Z_{\geq0}&=(\fw_{s}\cap \fg^Z_{>0})\oplus (\fv_{s}\cap \fg^Z_{\geq0})\,,\\
\fv_{s}\cap\fg^Z_{<0}&=(\fw_{t}\cap \fg^Z_{<0})\oplus (\fv_{t}\cap \fg^Z_{<0})\,.
\end{align}
This implies \eqref{=vw}. By Lemma \ref{lem:help} we have 
\begin{equation}\fn_s= \fv_s\oplus(\fg_{\varphi}\cap \fw_s)\subseteq \fv_s\oplus (\fw_s\cap \fg^{Z}_{\geq 0}),
\end{equation} and thus 
\begin{equation}\fr_s=\fv_{s}\oplus (\fw_{s}\cap \fg^Z_{>0})\oplus (\fw_0\cap \fg^Z_0\cap\fg_{\varphi}) \text{ and }\fr_{s}\cap (\fw_{t}\cap \fg_{\varphi})=\fw_{0}\cap \fg^Z_0\cap \fg_{\varphi}.
\end{equation}

Hence, \eqref{=vw} and \eqref{eq:lt} imply \eqref{=lrw}, and the rest is straightforward.
\end{proof}

Using Lemma \ref{lem:Z}, choose an $\sl_2$-triple $(e_\varphi,h,f_\varphi)$ in $\fg^Z_0$ such that $h$ commutes with $H$ and with $Z$, and $\varphi$ is given by the Killing form pairing with $f = f_\varphi$.  
Define $L_{t} = L_{H_t,\varphi} :=\Exp(\fl_{H_t,\varphi}(\mathbb{A})),\, R_{t}=R_{H_t,\varphi}:=\Exp(\fr_{H_t,\varphi}(\mathbb{A}))$ and $V_t := \Exp(\fv_t(\mathbb{A}))$. 
From Lemmas \ref{lem:key} and \ref{lem:StvN} we get the following.

\begin{lem}\label{lem:step}
  Let $(H,\varphi)$ and $(S,\varphi)$ be Whittaker pairs such that $(H,\varphi)$ dominates $(S,\varphi)$.
  Let $Z = S - H$, $H_t = H + tZ$ and define $R_t$, $L_t$ and $V_t$ as above.
\begin{enumerate}
\item \label{it:easy}
Let $t \geq 0$ and $\varphi'\in (\fg^*)^{H_t}_{>-2}$. Let also $I = L_t$ and $\mathfrak{j} = \fr_t$, or $I = R_t$ and $\mathfrak{j} = \fl_t$ respectively.
Let $A:=\Exp((\fu_t/\fv_t)^Z_{<0})(\A)$.
Then,
\begin{gather}
    \label{eq:FHtoFI}
  \cF_{H_t, \varphi, \varphi'}[\eta](g) = \quad\suml{\gamma \in \Exp(\mathfrak{j}/\fn_t)}\quad \cF_{H_t, \varphi, \varphi'}^{I}[\eta](\gamma g)  \, , \\
  \cF_{H_t, \varphi, \varphi'}^{I}[\eta](g) = \;\intl{[I/N_{H_t, \varphi}]}\; \cF_{H_t, \varphi, \varphi'}[\eta](ug) \, du \, , \\
  \label{eq:swap}
  \cF_{H_t,\varphi,\varphi'}^{L_t}[\eta](g) = \intl_{A} \cF_{H_t,\varphi,\varphi'}^{R_t}[\eta](vg) \, dv \,,
\end{gather}
\item \label{it:step}  
Now, let $0 \leq s < t$ such that there are no critical values in the interval $(s,t$) and let $\varphi'\in (\fg^*)^{H_t}_{>-2}\cap(\fg^*)^{H_s}_{>-2}$.
Then,
\begin{align}
  \label{eq:FLfromFR}
    \cF_{H_t,\varphi,\varphi'}^{L_t}[\eta](g) &= \intl_{[L_t/R_s]} \cF_{H_s,\varphi,\varphi'}^{R_s}[\eta](ug) \, du \\
  \label{eq:FRfromFL}
    \cF_{H_s,\varphi,\varphi'}^{R_s}[\eta](g) &= \sum_{\psi' \in \Psi'} \cF_{H_t,\varphi,\varphi'+\psi'}^{L_t}[\eta](g),
\end{align}
where $\Psi' := (\fg^*)^{H_t}_{-1}\cap (\fg^*)^{e_\varphi}\cap (\fg^*)^Z_{<0}$.

\item \label{it:step2} 
  Let $0 \leq s < t$ and $\varphi'$ as in~\ref{it:step}, and let $\psi\in (\fg^*)^{H_s}_{>-2}\cap(\fg^*)^{H_t}_{ -2}$. Then 
\begin{equation}
    \cF_{H_s,\varphi,\psi+\varphi'}^{R_s}[\eta](g) = \intl_{[R_s/V_t]} \sum_{\psi' \in \Psi''} \cF_{H_t,\varphi+\psi,\varphi'+\psi'}[\eta](ug) \, du,
\end{equation}
where $\Psi'' := (\mathfrak{g}^*)^{H_t}_{-1} \cap  (\mathfrak{g}^*)^{e_{\varphi+\psi}}$.

\end{enumerate}
\end{lem}

\begin{proof}
  For part \ref{it:easy} we note that $\fl_t$ and $\fr_t$ are isotropic subspaces of $\fu_t$ containing $\fn_t$ by Lemma~\ref{lem:key}\ref{it:MaxIs}. 
  Thus, the statement follows from Lemma~\ref{lem:StvN} and Corollary~\ref{cor:RootExchange}. The domain of summation for $\gamma$ in \eqref{eq:FHtoFI} follows from Lemma~\ref{lem:key}\ref{it:MaxIs}.
  Note that for Corollary~\ref{cor:RootExchange} we have that $\fl_t \cap \fr_t^\perp = \fl_t \cap \fr_t=\fn_t$ and $\fl_t/\fn_t\cong (\fu^{t}/\fv^t)^Z_{<0}$ by Lemma~\ref{lem:key}.
  
  For part \ref{it:step} we note from Lemma~\ref{lem:key}\ref{it:Emb} that $\fr_s \subseteq \fl_t$ and that $\fl_t/\fr_s$ is commutative with a set of representatives $\fg^{H_t}_1 \cap \fg_\varphi \cap \fg^Z_{>0}$.
  Indeed, $\fl_t/\fr_s$ projects naturally and isomorphically to $(\fg^{H_t}_1 \cap \fg_{\varphi})/(\fg^{H_t}_1 \cap \fg_\varphi \cap \fg^Z_0)$, and $\fg^{H_t}_1 \cap \fg_\varphi \subseteq \fg^Z_{\geq0}$ by Lemma~\ref{lem:help}.
  Thus, \eqref{eq:FLfromFR} follows directly by integration, while \eqref{eq:FRfromFL} follows from a Fourier expansion in $L_t/R_s$ with characters given by $(\fl_t/\fr_s)^*$ for which we have the set of representatives $\Psi' := (\fg^*)^{H_t}_{-1}\cap (\fg^*)^{e_\varphi}\cap (\fg^*)^Z_{<0}$.

  For part \ref{it:step2}, note first that $N'_t:=N_{H_t, \varphi+\psi}$ and $R_s = R_{H_s, \varphi}$ depend on different characters which means that we cannot directly relate them using Lemma~\ref{lem:key}\ref{it:Emb} as we did in part~\ref{it:step}.
  Instead, we will relate them both to $V_t = \Exp(\fv_t(\mathbb{A}))$ where $\fv_t := \fg^{H_t}_{>1}$ which does not depend on any character.
  
  We have that $\fv_t$ is an ideal in $\fl_t := \fl_{H_t, \varphi}$ with commutative quotient.
  Since $\fl_t$ and $\fr_s := \fr_{H_s, \varphi}$ are part of the same deformation with the same character we can then use \eqref{=lrw} in Lemma~\ref{lem:key}\ref{it:Emb} to see that $\fv_t$ is an ideal in $\fr_s$ with commutative quotient.
  Thus, 
  \begin{equation}
    \label{eq:part2-Vt}
    \cF^{R_s}_{H_s,\varphi,\psi+\varphi'}[\eta](g) = \intl{[R_s/V_t]} \cF_{V_t}[\eta](ug)\, du \quad \text{ where } \quad \cF_{V_t}[\eta](g) := \intl{[V_t]} \eta(vg)\chi_{\varphi+\psi+\varphi'}(v)^{-1} dv. 
  \end{equation}
 
  Furthermore, $\fv_t$ is an ideal in $\fn_t' := \fl_{H_t,\varphi+\psi}$ and $\fn_t'/\fv_t$ is commutative with a set of representatives $\mathfrak{g}^{H_t}_1 \cap  \mathfrak{g}_{\varphi+\psi}$ by \eqref{eq:N_Sphi}.
  The statement now follows from a Fourier expansion of $\cF_{V_t}[\eta](g)$ in $N_t'/V_t$ with characters given by $(\fn_t'/\fv_t)^*$ for which we have the representatives $\Psi'' := (\mathfrak{g}^*)^{H_t}_{-1} \cap (\mathfrak{g}^*)^{e_{\varphi+\psi}}$.
\end{proof}

\begin{lemma}
    \label{lem:conjugation-translation}
    Let $(S, \varphi, \psi)$ be a Whittaker triple, $\eta$ an automorphic function and $\gamma \in \Gamma$. Then,
    \begin{equation}
        \label{eq:conjugation}
        \cF_{S, \varphi, \psi}[\eta](g) = \cF_{\Ad(\gamma) S, \Ad^*(\gamma)\varphi, \Ad^*(\gamma)\psi}[\eta](\gamma g) \, .
    \end{equation}
\end{lemma}
\begin{proof}
    The proof is straightforward.
    We have that $\chi_{\varphi+\psi}(u) = \chi_{\Ad^*(\gamma)(\varphi + \psi)}(\Ad(\gamma) u )$. Indeed, the right-hand side equals
    \begin{equation}
        \chi\Bigl( \bigl(\Ad^*(\gamma)(\varphi + \psi) \bigr) (\Ad(\gamma) u ) \Bigr) = \chi\bigl((\varphi + \psi)(\Ad(\gamma^{-1}) \Ad(\gamma) u )\bigr) = \chi_{\varphi + \psi}(u)\,.
    \end{equation}
    We also have that $ \lie g^{\Ad(\gamma)S}_\lambda = \Ad(\gamma) \lie g^S_\lambda$ since, for $x \in \lie g$, $[\Ad(\gamma) S, \Ad(\gamma) x] = \Ad(\gamma)[S,x]$. Similarly, $ \lie g_{\Ad^*(\gamma)\varphi} = \Ad(\gamma) \lie g_\varphi$ and thus, $\lie n_{\Ad(\gamma) S,\Ad^*(\gamma) \varphi} = \Ad(\gamma) \lie n_{S,\varphi}$.
    
    Hence, using the automorphic invariance of $\eta$, the right-hand side of \eqref{eq:conjugation} equals
    \begin{equation}
      \intl{[N_{\Ad(\gamma)S,\Ad^*(\gamma)}]} \eta(\gamma^{-1} u \gamma g) \chi_{\Ad^*(\gamma)(\varphi + \psi)}(u)^{-1} \, du = \intl{[\gamma^{-1} (N_{\Ad(\gamma)S,\Ad^*(\gamma)} )\gamma]} \eta(u' g) \chi_{\Ad^*(\gamma)(\varphi + \psi)}(\Ad(\gamma) u')^{-1} \, du' \,.
    \end{equation}
By the arguments above, this equals $\cF_{S, \varphi, \psi}[\eta](g)$.
\end{proof}

\section{The reduction algorithm}
\label{sec:levi-distinguished}

\setcounter{lemma}{0}

\subsection{Statement of algorithm}\label{subsec:alg}

Before stating Algorithm~\ref{alg:domin}, we first start with an algorithm that expresses any quasi-Fourier coef\-ficient $\cF_{H,\varphi,\varphi'}$
in terms of Fourier coefficients $\cF_{S,\psi}$ such that either $(H,\varphi)$ dominates $(S,\psi)$ or 
$\Gamma \psi >\Gamma \varphi$.

\begin{algorithm}\label{alg:quasi}
Given a Whittaker triple $(H,\varphi,\varphi')$, choose $h$ that commutes with $H$ such that $(h, \varphi)$ is a neutral pair. 
This is possible by Lemma~\ref{lem:Z}. Let $Z = H - h$, and let $H_t:=H+tZ$ for any $t\in \Q$.
Let $m_Z$ denote the minimal positive eigenvalue of $Z$ and $M_h$ denote the maximal positive eigenvalue of $h$. 
Let $T:=(M_h+2)/m_Z$. 
If $Z$ is central we set $m_Z$ to 1 for definiteness.
Then for every $i\geq 0$ we have
\begin{equation}\label{=Tbig}
\fg^{H_T}_{\geq i}=(\fg^h_{\geq i})^Z_0\oplus (\fg^{H_T}_{\geq i})^Z_{>0} , \quad
\fg^Z_{>0}\subset \fg^{H_T}_{> 2} \text{ and } \fg^Z_{<0}\subset \fg^{H_T}_{< -2}.
\end{equation}
Since $\fg_{\varphi}\subset \fg^h_{\leq 0}$, and since $\ad(h)$ has integer eigenvalues, we obtain
\begin{equation}\label{=Tbig_phi} 
\fn_{H_T,\varphi}=\fv_T=\fg^{H_T}_{>1}=\fg^{H_T}_{\geq 2}.
\end{equation}

We get two cases depending on whether there are any quasi-critical $t$ in the interval~$(0,T)$. Recall that the number of quasi-critical $t$ is in any case finite.

\begin{enumerate}[align=left, itemindent=1.5\parindent, leftmargin=0pt]
\item \label{it:EasyCase} If there are no quasi-critical $t\in (0,T)$ then, using Lemma \ref{lem:step}\ref{it:easy} and~\ref{it:step}, we may express $\cF_{H,\varphi,\varphi'}$ linearly in terms of the set of all $\cF_{H_T,\varphi,\varphi'+\psi'}$ with $\psi' \in \Psi' = (\lie g^*)^{H+TZ}_{-1} \cap (\lie g^*)^{e_\varphi} \cap (\lie g^*)^{Z}_{<0}$. 
By \eqref{=Tbig}
we have $\Psi'=\{0\}$ and $L_T=N_{H_T,\varphi}$. By \eqref{=Tbig_phi}, $\varphi'$ vanishes on~$\fn_{H_T,\varphi}$, thus $\cF^{L_T}_{H_T,\varphi,\varphi'}=\cF_{H_T,\varphi,\varphi'}=\cF_{H_T,\varphi}$. Altogether we have
\begin{multline}\label{=ES}
\cF_{H, \varphi, \varphi'}[\eta](g) = \quad\suml{\gamma \in \Exp(\fl_0 / \fn_0)}\quad\!\!\! \cF_{H, \varphi, \varphi'}^{R_0}[\eta](\gamma g) =\quad\suml{\gamma \in \Exp(\fl_0 / \fn_0)}\quad\!\!\! \cF_{H_T, \varphi, \varphi'}^{L_T}[\eta](\gamma g) = \quad\suml{\gamma \in \Exp(\fl_0 / \fn_0)}\quad\!\!\! \cF_{H_T, \varphi}[\eta]( \gamma g).
\end{multline}

\item \label{it:DiffCase} Now assume that there are quasi-critical numbers in $(0,T)$ and let $s$ be the smallest one.
Since $s$ is the first quasi-critical value we have that $(\lie g^*)^{H}_{>-2} \subseteq (\lie g^*)^{H_s}_{\geq-2}$ because this is the first point where something new may enter the $-2$-eigenspace.
Decompose $\varphi' = \psi + \varphi'' $ where $\psi \in (\lie g^*)^{H_s}_{-2}$ and $\varphi'' \in (\lie g^*)^{H_s}_{>-2}$.
Using Lemma \ref{lem:step}\ref{it:easy} and~\ref{it:step2}, we may express $\cF_{H,\varphi,\varphi'}$  in terms of the Fourier coefficients $\cF_{H_{s},\varphi+\psi,\varphi'' + \psi''}$ with $\psi''\in \Psi'' = (\mathfrak{g}^*)^{H_s}_{-1} \cap (\mathfrak{g}^*)^{e_{\varphi+\psi}}$ in the following way:
\begin{equation}\label{=DS}
    \cF_{H,\varphi,\varphi'}[\eta](g) = \quad\suml{\gamma \in \Exp(\fl_0/\fn_0)}\quad \cF_{H,\varphi,\varphi'}^{R_0}[\eta](\gamma g) = \!\!\! \sum_{\substack{\gamma \in \Exp(\fl_0/\fn_0) \\ \psi''\in \Psi''}} \, \intl_{[R_0/V_s]} \!\!\! \cF_{H_{s},\varphi+\psi,\varphi'' + \psi''}[\eta](u\gamma g) \, du.
\end{equation}

Now, we repeat the procedure for each triple $\cF_{H_{s},\varphi+\psi,\varphi''+\psi''}$ and so on until we reach an expression that includes only Fourier coefficients.  \hfill$\blacksquare$
\end{enumerate}
\end{algorithm}

\begin{lem}\label{lem:quasiTer}
Algorithm \ref{alg:quasi} terminates in a finite number of steps, and expresses  $\cF_{H,\varphi,\varphi'}$
in terms of Fourier coefficients $\cF_{S_i,\Phi_i}$ such that for each $i$, either $\Phi_i=\varphi$ and $(H,\varphi)$ dominates $(S_i,\Phi_i)$,  or 
$\Gamma \Phi_i >\Gamma \varphi$. 
\end{lem}

We postpone the proof of the above lemma to the next subsection.
Now let $(H,\varphi)$ and $(S,\varphi)$ be Whittaker pairs such that $(H,\varphi)$ dominates $(S,\varphi)$. The following algorithm expresses  $\cF_{H,\varphi}$ in terms of  $\cF_{S,\varphi}$  and 
Fourier coefficients $\cF_{S',\psi}$ such that  
$\Gamma \psi >\Gamma \varphi$. 
\setcounter{maintheorem}{0} 
\begin{mainalgorithm}\label{alg:domin}
Let $Z:=S-H$ and $H_t:=H+tZ$.
Let first $s = 0$ and $t$ be the first quasi-critical point in the interval $(s,1)$. If there are no critical points in this interval we let $t = 1$ which will be the end point of this algorithm.
Using Lemma \ref{lem:step}\ref{it:easy} and~\ref{it:step}, we have that $\cF_{H, \varphi}[\eta](g)$ equals
\begin{align}\label{=EDomin}
\cF_{H_s, \varphi}[\eta](g) &= \quad\suml{\gamma \in \Exp(\fl_s / \fn_s)}\quad \cF_{H, \varphi}^{R_s}[\eta](\gamma g) =
\sum_{\substack{\gamma \in \Exp(\fl_s/\fn_s) \\ \psi'\in \Psi'}}
\cF_{H_t, \varphi, \psi'}^{L_t}[\eta](\gamma g) \\ 
&\hspace{-10mm}=  \sum_{\substack{\gamma \in \Exp(\fl_s/\fn_s) \\ \psi'\in \Psi'}}\,\int\limits_{[L_t/N_{H_t,\varphi}]} \hspace{-1em} \cF_{H_t, \varphi, \psi'}[\eta](\gamma ug) du 
= \hspace{-0.5em} \sum_{\substack{\gamma \in \Exp((\fg^{H_s}_1)^Z_{<0}) \\ \psi'\in \Psi'}}\,\,\int\limits_{[\Exp((\fg^{H_t}_1)^Z_{<0}(\A))]}\hspace{-2em} \cF_{H_t, \varphi, \psi'}[\eta](\gamma ug) du,\nonumber
\end{align}
where $\Psi' = (\lie g^*)^{H_t}_{-1} \cap (\lie g^*)^{e_\varphi} \cap (\lie g^*)^{Z}_{<0}$.
The last step follows from Lemma~\ref{lem:key}\ref{it:MaxIs}.
For all non-zero $\psi'\in \Psi'$ we run Algorithm \ref{alg:quasi} to express the quasi-Fourier coefficient 
$\cF_{H_t,\varphi,\psi'}$ in terms of Fourier coefficients, that, as we will show, correspond to higher orbits.
For the $\psi'=0$ term we iterate the same step as above using \eqref{=EDomin} and Algorithm~\ref{alg:quasi} but now with $s = t$ and $t$ being the next quasi-critical point, until we reach $t = 1$.

\end{mainalgorithm}

\begin{prop}\label{prop:dominTer}
    Algorithm \ref{alg:domin} terminates in a finite number of steps and produces an expression of the form
    \begin{equation}
                \cF_{H,\varphi}[\eta](g) = \mathcal{M}_H^S(\mathcal{F}_{S,\varphi}[\eta]) + \mathcal{A}, 
\end{equation}
where  
\begin{equation}
    \mathcal{M}_H^S(\mathcal{F}_{S,\varphi}[\eta]) = \sum_{w\in \Omega}\,\,\, 
            \intl_{V} \intl_{ {[ U]}} \cF_{S,\varphi}[\eta](wvug) \, dudv \,
\end{equation}
which is called the ``main term'' in the introduction,
\begin{equation}
\label{UVW}
\begin{gathered}
 U=\Exp(\fu(\A)),\quad V=\Exp(\fv(\A)),\quad \Omega=\Exp(\fw(\K)), \\
  \fu:= (\fg^{H}_{>1})^S_1\cong (\fg^{H}_{>1})^S_{\geq1}/(\fg^{H}_{>1})^S_{>1}, \quad
  \fv:= (\fg^{H}_{>1})^S_{<1}\cong\fg^H_{> 1}/(\fg^{H}_{>1})^S_{\geq 1}, \\
  \fw:=(\fg^{H}_{1})^S_{<1}\cong \fg^H_{\geq 1}/\left(\fg^H_{> 1}+ (\fg^{H}_1)^S_{\geq 1}\right),
\end{gathered}
\end{equation}
  and $\mathcal{A}$ is a countable absolutely convergent sum of integral transforms of Fourier coefficients corresponding to orbits bigger than that of $\varphi$. 
\end{prop}
We prove this proposition in the next subsection.

\subsection{Proofs of correctness and termination algorithms, and Theorem~\ref{thm:IntTrans}}\label{subsec:Pfs}

We prove the following lemma in Appendix~\ref{sec:Geo}.
\begin{lemma}\label{lem:SameHOrbit}
    Let $(H,\varphi)$ be a Whittaker pair, and let $Z$ be a rational semi-simple element that commutes with $H$ and with $\varphi$. Let $\varphi'\in (\fg^*)^Z_{>0}\cap (\fg^*)^{S}_{2}$. 
    Then either $\dim {\bf G}(\C)(\varphi+\varphi')>\dim {\bf G}(\C)(\varphi)$, or 
 $\varphi$ is conjugate to $\varphi+\varphi'$ by the stabilizer of $H$ in $\Gamma$.
\end{lemma}

Let us now present the proofs that were postponed from the last subsection.

\begin{proof}[Proof of Lemma \ref{lem:quasiTer}]
In case \ref{it:EasyCase} of the algorithm it terminates in one step, and \eqref{=ES} expresses $\cF_{H,\varphi,\varphi'}$ in terms of $\cF_{H_T,\varphi}$, and $(H,\varphi)$ dominates $(H_T,\varphi)$. Thus assume that we are in case \ref{it:DiffCase}, and let $s\in (0,T]$ be the smallest critical number. 
Decompose $\varphi' = \psi + \varphi'' $ where $\psi \in (\lie g^*)^{H_s}_{-2}$ and $\varphi'' \in (\lie g^*)^{H_s}_{>-2}$.
Note that $\psi\in (\fg^*)^{Z}_{<0}$ and thus the orbit $\Gamma(\varphi+\psi)$ of $\varphi+\psi$ is bigger than or equal to the orbit $\Gamma\varphi$ of $\varphi$.

If $\Gamma(\varphi+\psi)=\Gamma\varphi$ then, by Lemma~\ref{lem:SameHOrbit}, 
$\varphi+\psi$ is conjugate to $\varphi$ under the stabilizer of $H_s$ in $\Gamma$. 
Then, by shifting the argument, we express each $\cF_{H_s,\varphi+\psi,\varphi''+\psi''}$ in terms of   $\cF_{H_s,\varphi,\varphi'''}$ for some $\varphi'''$.
Then we run the algorithm on $\cF_{H_s,\varphi,\varphi'''}$ and it terminates by induction on the finite number of critical values in the interval $(0,T)$.

If $\Gamma(\varphi+\psi)>\Gamma\varphi$ then, by Lemma~\ref{lem:SameHOrbit}, $\dim {\bf G}(\C)(\varphi+\psi)>\dim {\bf G}(\C)(\varphi)$, and thus the algorithm terminates by induction on $\dim {\bf G}(\C)(\varphi)$. 
\end{proof}

\begin{lem}\label{lem:quasi2high}
Let $(H,\varphi)$ be a Whittaker pair, 
and let $(e_{\varphi},h,f_{\varphi})$ be an $\sll_2$-triple such that $\varphi$ is given by Killing form pairing with $f_{\varphi},$ and $h$ commutes with $H$. Let 
$0 \neq \varphi' \in (\fg^*)^{H}_{>-2}\cap (\fg^*)^{H}_{\leq -1} \cap(\fg^*)^{e_\varphi}$. 
Then Algorithm \ref{alg:quasi} expresses the quasi-Fourier coefficient $\cF_{H,\varphi,\varphi'}$ in terms of Fourier coefficients $(S_i,\Phi_i)$ with $\Gamma \Phi_i>\Gamma \varphi$.
\end{lem} 

\begin{proof}
As in the algorithm, we let $Z:=H-h$, and for any $t\geq 0$ denote $H_t:=H+tZ$. Then $\varphi'$
decomposes to a sum of eigenvectors $\psi_i$ of $Z$, and each of these lies in $((\fg^*)^{e_{\varphi}})^H_{\leq -1}$. 
Since $(\fg^*)^{e_{\varphi}}\subset (\fg^*)^{h}_{\geq 0}$, we obtain $\psi_i\in ((\fg^*)^{h}_{\geq 0})^{Z}_{<0}$. 
Hence, while running Algorithm~\ref{alg:quasi}, some of the eigenvectors $\psi_i$ will join the space $(\fg^*)^{H_t}_{-2}$ for some $t>0$, 
leading to an orbit that is greater than, or equal to, that of $\varphi$.
To show that it is not equal to $\Gamma \varphi$ note that $\varphi+\psi_i$ lies in the Slodowy slice $\varphi+(\fg^*)^{e_{\varphi}}$, and thus its complex orbit has bigger dimension than that of $\varphi$. 
More Whittaker triples will be produced while running Algorithm \ref{alg:quasi},  but their third elements will lie in $(\fg^*)^{H_t}_{-1} \cap(\fg^*)^{e_\varphi}$ for some $t>0$ and thus lead to bigger orbits using the argument as above.
\end{proof}

\begin{proof}[Proof of Proposition \ref{prop:dominTer}]
As in the statement of Algorithm~\ref{alg:domin} let $Z:=S-H$ and $H_t:=H+tZ$.
If there are no quasi-critical points in the interval $(0,1)$ then $t=1$ and the algorithm terminates in one step.
In this case we can decompose the expression in the right-hand side of \eqref{=EDomin} as
\begin{align}
\sum_{\gamma \in \Exp((\fg^H_1)^S_{<1}) }\int_{[\Exp((\fg^S_1)^H_{>1}(\A))]} \cF_{H_T, \varphi, \psi'}[\eta](\gamma ug) du + \mathcal{A},
\end{align}
where 
\begin{align}
\mathcal{A}=\sum_{\substack{\gamma \in \Exp((\fg^H_1)^S_{<1}) \\ 0\neq \psi'\in \Psi'}}\int_{[\Exp((\fg^S_1)^H_{>1}(\A))]} \cF_{H_T, \varphi, \psi'}[\eta](\gamma ug) du\,.
\end{align}
By Lemma~\ref{lem:quasi2high}, for every non-zero $\psi'\in \Psi'=(\lie g^*)^{S}_{-1} \cap (\lie g^*)^{e_\varphi} \cap (\lie g^*)^{Z}_{<0}$
Algorithm~\ref{alg:quasi} expresses $\cF_{S,\varphi,\varphi+\psi'}$ in terms of Fourier coefficients $(S_i,\Phi_i)$ with $\Gamma \Phi_i>\Gamma \varphi$. Thus the term $\mathcal{A}$ satisfies the conditions of the proposition. 

Let us now assume that there exist quasi-critical points $0<t_1<\cdots <t_{n}<1$.
Lemma \ref{lem:quasi2high} again implies that all non-zero $\psi'\in \Psi''_i=(\lie g^*)^{H_{t_i}}_{-1} \cap (\lie g^*)^{e_\varphi} \cap (\lie g^*)^{Z}_{<0}$ will lead to Fourier coefficients corresponding to bigger orbits, that will be accumulated in~$\mathcal{A}$. 
It is left to track down the formula for the term that we get when we take all $\psi'$ in all the steps to be zero. Let
\begin{equation}
    \fv_{i}:=(\fg^{H_{t_i}}_{\geq 1}\cap \fg^{Z}_{<0})/(\fg^{H_{t_i}}_{>1}\cap \fg^{Z}_{<0}) \quad \text{ and } \quad V_i=\Exp(\fv_i) \, . 
\end{equation}
By Lemma~\ref{lem:step}\ref{it:easy} we have
\begin{equation}
    \cF^L_{H_{t_i},\varphi}[\eta](g)=\intl{V_i}\cF^R_{H_{t_{i}},\varphi}[\eta](v_ig) \, dv_i\,.
\end{equation}
Using  Lemma~\ref{lem:step}\ref{it:step} (retaining only $\psi' = 0$) we obtain 
\begin{equation}\label{=RInt2}
    \cF^R_{H,\varphi}[\eta](g)=\intl_{V_1}\dots\intl_{V_{n-1}} \intl_{V_n} \cF^L_{S,\varphi}[\eta](v_n\dots v_1g) dv +\mathcal{A}.
\end{equation}

Since $\fv=\bigoplus_{i=1}^n(\fg^{H_{t_i}}_1\cap \fg^{Z}_{<0})$, and as a commutative Lie algebra $\fg^{H_{t_i}}_1\cap \fg^{Z}_{<0}$ is naturally isomorphic to $\fv_i$, the group $V$ is glued from $V_i$. Thus 

\begin{equation}\label{=Fubini}
    \intl_{V_1}\dots\intl_{V_{n-1}} \intl_{V_n} \cF^L_{S,\varphi}[\eta](v_n\dots v_1g) dv = 
    \intl{V} \cF^L_{S,\varphi}[\eta](vg) \, dv \, .
\end{equation}
The proposition now follows from Lemma~\ref{lem:step}\ref{it:easy}. 
\end{proof}

\begin{proof}[Proof of Theorem \ref{thm:IntTrans}]
Under the assumption of Theorem \ref{thm:IntTrans}, for any $\Phi$ with $\Gamma \Phi>\Gamma \varphi,$ the neutral coefficient $\cF_{h,\Phi}[\eta]$ vanishes. By \cite[Theorem C]{GGS}, this implies the vanishing of $\cF_{S,\Phi}[\eta]$ for any pair $(S,\Phi)$ with $\Gamma \Phi>\Gamma \varphi$. Thus, the term $\mathcal{A}$ in Proposition \ref{prop:dominTer} also vanishes. The theorem follows.
\end{proof}

\subsection{Additional results}\label{subsec:more}
Let us prove for future use some additional results that utilise the technique of this section.
For the entire subsection we fix Whittaker pairs $(H,\varphi)$ and $(S,\varphi)$  such that $(H,\varphi)$ dominates $(S,\varphi)$.
Let $Z:=S-H$ and let $H_t:=H+tZ$.
Let $0<t_1<\dots<t_{n}<1$ be all the critical values between $0$ and $1$.
Let $t_0:=0$ and $t_{n+1}:=1$. Lastly,  for each $t_i$, let $R:=R_{t_i}$ and $L:=L_{t_i}$ be defined as in \eqref{=lt}.

\begin{lem}\label{lem:MaxStep}
 Let $\eta$ be an automorphic function such that $\Gamma \varphi\in \WS(\eta)$. Then 
we have $\cF^R_{H_{t_i},\varphi}[\eta]=\cF^L_{H_{t_{i+1}},\varphi}[\eta]$.
\end{lem}
\begin{proof}
Denote $H_{j}:=H_{t_{j}}$ for any $j$, and $\fc:=(\fg^*)^{H_{i+1}}_{-1}\cap (\fg^*)^{e_{\varphi}}\cap (\lie g^*)^Z_{<0}$.
Arguing as in the proof of Lemma \ref{lem:step}\ref{it:step}, we obtain \begin{align}\cF^R_{H_{i},\varphi}[\eta]=\sum_{\varphi'\in \fc }\cF^L_{H_{i+1},\varphi,\varphi'}[\eta].
\end{align}
We have to show that for any non-zero $\varphi'\in \fc$, we have $\cF^L_{H_{i+1},\varphi,\varphi'}[\eta]=0$.
But $\cF^L_{H_{i+1},\varphi,\varphi'}[\eta]$ is an integral of $\cF_{H_{i+1},\varphi,\varphi'}[\eta]$, which 
by Lemma \ref{lem:quasi2high} is expressed through coefficients $\cF_{S_j,\Phi_j}$ with $\Gamma \Phi_j>\Gamma\varphi$.
Since $\varphi\in \WS(\eta),$ $\cF_{S_j,\Phi_j}[\eta]\equiv 0$ for all $j$, and thus $\cF^L_{H_{i+1},\varphi,\varphi'}[\eta]\equiv0$.
\end{proof}

\begin{cor}\label{thm:BetterIntTrans}
 Let $\eta$ be an automorphic function such that $\Gamma \varphi\in \WS(\eta)$.
Let $\fv:= \fg^H_{> 1}/(\fg^H_{> 1}\cap \fg^{S}_{\geq 1})$, $\fv':= \fg^S_{> 1}/(\fg^S_{> 1}\cap (\fg^{H}_{\geq 1}+\fg_{\varphi}))$, $V:=\Exp(\fv(\A))$, and $V':=\Exp(\fv'(\A))$.
Then, 
\begin{equation}
  \cF^R_{H,\varphi}[\eta](g)=\int_V  \cF^L_{S,\varphi}[\eta](vg)dv\quad \text{ and }\quad \cF^L_{H,\varphi}[\eta](g)=\int_{V'}  \cF^R_{S,\varphi}[\eta](vg)dv\,.
\end{equation}
\end{cor}

\begin{proof}
With the notation introduced above, let
\begin{equation}
    \fv_{i}:=(\fg^{H_{t_i}}_{\geq 1}\cap \fg^{Z}_{<0})/(\fg^{H_{t_i}}_{>1}\cap \fg^{Z}_{<0}) \quad \text{ and } \quad V_i=\Exp(\fv_i) \, . 
\end{equation}
By Lemma~\ref{lem:step}\ref{it:easy} we have
\begin{equation}
    \cF^L_{H_{t_i},\varphi}[\eta](g)=\intl{V_i}\cF^R_{H_{t_{i}},\varphi}[\eta](v_ig) \, dv_i\,.
\end{equation}
By Lemma \ref{lem:MaxStep} we have $\cF^R_{H_{t_i},\varphi}[\eta]=\cF^L_{H_{t_{i+1}},\varphi}[\eta]$. Thus 
\begin{equation}\label{=Fubini2}
    \cF^R_{H,\varphi}[\eta](g)=\intl_{V_1}\dots\intl_{V_{n-1}} \intl_{V_n} \cF^L_{S,\varphi}[\eta](v_n\dots v_1g) dv = 
    \intl{V} \cF^L_{S,\varphi}[\eta](vg) \, dv \, .
\end{equation}

Similarly, let 
\begin{equation}
    \fv'_{i}:=(\fg^{H_{t_i}}_{\geq 1})^{Z}_{>0}/(\fg_{\varphi}\cap\fg^{H_{t_i}}_{1}+ (\fg^{H_{t_i}}_{>1})^{Z}_{>0})\cong (\fg^{H_{t_i}}_{>1})^Z_{>0}/((\fg^{H_{t_i}}_{>1})^Z_{>0}\cap \fg_{\varphi})   \text{ and }  V'_i=\Exp(\fv'_i) \, . 
\end{equation}
Then the nilpotent group $V'$ is glued from the commutative groups $V'_i$ and 
\begin{equation}\label{=Fubini3}
  \cF^L_{S,\varphi}[\eta](g)=\intl_{V'_{n}}\dots\intl_{V'_{2}} \intl_{V'_1} \cF^R_{H,\varphi}[\eta](v_1\dots v_{n}g) dv = 
    \intl{V'} \cF^L_{H,\varphi}[\eta](vg) \, dv \, . \qedhere
\end{equation}
\end{proof}

\begin{prop}[{cf. \cite[Theorem A]{GGS} for a local analogue}]\label{prop:domin}
    For any  $\eta \in C^{\infty}(\Gamma \backslash G)$ with $\cF_{H,\varphi}[\eta]\equiv 0$ we have $\cF_{S,\varphi}[\eta]\equiv 0$. 
\end{prop}

\begin{proof}
  Let $Z:=S-H$, and for any $t\geq 0$ let $H_t:=H+tZ$. Let $0 < t_1 < \dots < t_k < 1$ be all the critical values of $t$ between 0 and 1. Let $t_0:=0$ and $t_{k+1}:=1$. By Lemma \ref{lem:step}\ref{it:easy} and~\ref{it:step}, for any $0\leq i\leq k$, $\cF_{H_{t_{i+1}},\varphi}$ is expressed in terms of $\cF_{H_{t_i},\varphi}$.
  Since $H_{t_0}=H$, we obtain by induction that $\cF_{H_{t_i},\varphi}[\eta]\equiv 0$ for all $i$.
Since $H_{t_{k+1}}=S$, the proposition follows.
\end{proof}

\subsection{Levi-distinguished coefficients}\label{subsec:LeviDist}
Let us show that any Whittaker pair  $(H,\varphi)$ dominates a Levi-distinguished Whittaker pair.
Using Lemma \ref{lem:Z}, decompose $H=h+Z$, where $(h,\varphi)$ is a neutral pair, and $Z$ commutes with $h$ and $\varphi$.

\begin{notation}\label{not:z}
Let $C\subseteq \Gamma$ denote the centralizer of $(h,\varphi)$. Let $A$ denote a maximal split torus of $C$ such that its Lie algebra $\fa$ includes $Z$, and let $M$ denote the centralizer of $\fa$ in $G$. Then $M$ is a Levi subgroup of $G$, $\fm$ includes $h,Z$ and $\varphi$, and $\varphi$ is $\K$-distinguished in $\fm$. Let $z$ be a rational semi-simple element of $\fa$ that is generic in the sense that its centralizer is $M$.
\end{notation}

\begin{lemma} \label{lem:K-distinguished}
As an element of $\fm$, $\varphi$ is $\K$-distinguished.
\end{lemma}
\begin{proof}    
Let $\fl$ be the Lie algebra of a Levi subgroup of $M$ defined over $\K$ such that $\varphi\in \fl^*$. We have to show that $L=M$. By replacing $L$ by its conjugate we can assume $h\in \fl$, and that there exists a rational semi-simple element $z'\in \fm$ such that $\fl$ is the centralizer of $z'$. Then $z'$ commutes with $h$ and $\varphi$ and we have to show that $z'$ is central in $\fm$.  

Indeed, $z'\in \fm\cap \fc=\fa$. Now, any $X\in \fm$ commutes with $z$, and thus with any element of $\fa$, since $z$ is generic in $\fa$. Thus $\fa$  lies in the center of $\fm$ and thus $z'$ is central.
\end{proof}

Note that the eigenvalues of the adjoint action of any Lie algebra element are symmetric around zero.

\begin{notation}\label{not:Zprime}
Let $N$ be a positive integer that is bigger than the ratio of the maximal eigenvalue of $\ad(z)$ by the minimal positive eigenvalue of $\ad(Z)$. Let
\begin{equation}\label{=:Z'}
Z':=NZ+z. 
\end{equation}
\end{notation}
From our choice of $N$ we have 
\begin{equation}\label{=Z'}
\fg^{Z'}_{>0}=\fg^{Z}_{>0}\oplus(\fg^{Z}_{0}\cap \fg^{z}_{>0}) \text{ and } \fg^{Z'}_{0}=\fg^{z}_{0}=\fm\subseteq \fg^{Z}_{0}.
\end{equation}
That $\lie m \subseteq \lie g^Z_0$ follows from the fact that $M$ is the centralizer of $z$ which equals the centralizer of $\lie a$ and $\lie a$ includes $Z$.

\begin{lemma}\label{lem:Z'main}
For rational $T > 0$, $(H, \varphi)$ dominates $(H+TZ', \varphi)$, that is, $H,\varphi$ and $T Z'$ commute, and satisfy \eqref{=domin}.
\end{lemma}
\begin{proof}
    By construction $H = h + Z$, $\varphi$ and $Z$ commute, and since $h, Z, \varphi \in \lie m$ they commute with $z$. Thus, $Z'$ commutes with $H$ and $\varphi$. Furthermore, 
$\fg_{\varphi}\cap \fg^H_{\geq 1}\subseteq \fg^h_{\leq 0}\cap \fg^H_{\geq 1}\subseteq  \fg^Z_{>0} \subseteq \fg^{Z'}_{>0} = \fg^{TZ'}_{>0}.$
\end{proof}

\begin{lemma}\label{lem:LeviDist}
For a fixed $\lambda \in \Q$, and a rational $T > 0$ large enough, 
\begin{equation}\label{=LeviDist}
\fg^{H+TZ'}_{> 1}=\fg^{H+TZ'}_{\geq 2}=\fg^{Z'}_{>0}\oplus(\fg^{Z'}_{0}\cap \fg^{H+TZ'}_{>1})=\fg^{Z'}_{>0}\oplus \fm^{h}_{\geq 2}\text{ and }\fg^{H+TZ'}_{ \lambda}=\fm^{H}_{\lambda}=\fm^{h}_{\lambda}.
\end{equation}
The Fourier coefficient $\mathcal{F}_{H+TZ', \varphi}$ is then Levi-distinguished.
\end{lemma}

\begin{proof}
    For large enough $T$, we have that $\lie g^{H+TZ'}_{>1} \cap \lie g^{Z'}_{<0} = \{0\}$ and $\lie g^{H+TZ'}_{>1} \cap \lie g^{Z'}_{>0} = \lie g^{Z'}_{>0}$. Thus $\lie g^{H+TZ'}_{>1} = \lie g^{H+TZ'}_{>1} \cap \bigl( \lie g^{Z'}_{<0} \oplus \lie g^{Z'}_0 \oplus \lie g^{Z'}_{>0} \bigr) = \lie g^{Z'}_{>0} \oplus \bigl( \lie g^{Z'}_0 \cap \lie g^{H+TZ'}_{>1} \bigr)$. Since $H = h + Z$ and $\lie g^{Z'}_0 = \lie m \subseteq \lie g^Z_0$ we have that $\lie g^{Z'}_0 \cap \lie g^{H+TZ'}_{>1} = \lie g^{Z'}_0 \cap \lie g^{h}_{>1}$ and since $h$ is neutral $\lie g^{h}_{>1} = \lie g^{h}_{\geq 2}$. Now $\lie g^{Z'}_0 = \lie m$ and thus, $\lie g^{H+TZ'}_{>1} = \lie g^{Z'}_{>0} \oplus \bigl( \lie g^{Z'}_0 \cap \lie g^{h}_{\geq2} \bigr) = \lie g^{Z'}_{>0} \oplus \lie m^h_{\geq2}$. 
    Doing the same manipulations for $\lie g^{H+TZ'}_{\geq2}$ one ends up with the same result, proving the equality $\lie g^{H+TZ'}_{>1} = \lie g^{H+TZ'}_{\geq2}$. 
 
    Now, for any fixed $\lambda \in \Q$ and a large enough $T$, we have that $\lie g^{H+TZ'}_\lambda = \lie g^{H}_\lambda \cap \lie g^{Z'}_0 = \lie g^{H}_\lambda \cap \lie m = \lie m^H_\lambda$. Again, since $H = h + Z$ and $\lie m \subseteq \lie g^Z_0$, we get that $\lie m^H_\lambda = \lie m^h_\lambda$.

    Since $H + TZ' = h + Z + TZ'$, the semi-simple element denoted by $Z$ in 
    Definition~\ref{def:Levi-distinguished} is here $Z + TZ'$, which, for large enough $T$ has the centralizer $\lie g^Z_0 \cap \lie g^{Z'}_0 = \lie g^{Z'}_0 = \lie m$. By Lemma~\ref{lem:K-distinguished}, $\varphi$ is $\K$-distinguished in $\lie m$. Since $\lie g^Z_{>0} \subseteq \lie g^{Z'}_{>0}$ we have that $\lie g^{Z+TZ'}_{>0} = \lie g^{Z'}_{>0}$ and thus \eqref{=LeviDist} implies \eqref{=LeviDist0} which means that $\cF_{H+TZ', \varphi}$ is Levi-distinguished.
\end{proof}

\begin{cor}\label{cor:LeviDistDomin}
    Any Whittaker pair  $(H,\varphi)$ dominates a Levi-distinguished Whittaker pair.
\end{cor}

\begin{cor}\label{cor:algGood}
    Algorithm \ref{alg:domin} allows us to express any Fourier coefficient $\cF_{H,\varphi}$ in terms of Levi-distinguished Fourier coefficients with characters in orbits which are equal or bigger than $\Gamma\varphi$.
\end{cor}
\begin{proof}
Choose a Levi-distinguished Whittaker pair $(S,\varphi)$ dominated by $(H,\varphi)$. Then  Algorithm \ref{alg:domin} expresses 
$\cF_{H,\varphi}$ in terms of $\cF_{S,\varphi}$, and Fourier coefficients $\cF_{H'_i,\Phi_i}$ corresponding to higher orbits. Each of the pairs $(H'_i,\Phi_i)$ dominates a Levi-distinguished Whittaker pair $(S_i,\Phi_i)$. We repeat the procedure for each pair. The process terminates in a finite number of steps since the dimension of each complex orbit ${\bf G}(\C)\Phi_i$ is bigger than that of ${\bf G}(\C)\varphi$. 
\end{proof}

\begin{remark}\label{rem:alg}
\begin{enumerate}
\item Note that $\eta=\cF_{0,0}[\eta]$. Thus, the algorithm allows us to express any automorphic function in terms of its Levi-distinguished Fourier coefficients.
\item Algorithm \ref{alg:domin} produces a general formula, that holds for all automorphic functions $\eta \in C^{\infty}(\Gamma\backslash G)$.  However, if we put additional assumptions on $\eta$ the algorithm might terminate earlier and produce a shorter expression.
\item \label{itm:levidist-PL} By Lemma \ref{lem:WhitPL}, the Levi-distinguished Fourier coefficients of PL elements are Whittaker coefficients. This implies that if all Levi-distinguished Fourier coefficients of some automorphic function $\eta$ corresponding to non-PL orbits vanish, then the algorithm allows us to express $\eta$ in terms of its Whittaker coefficients.

\item Vanishing as in~\ref{itm:levidist-PL} happens in two important cases. One is the case of $\GL_n$ in which all orbits are PL orbits. We explain the results in this case in \S\ref{subsec:GL} below. Another case is the case when $G$ is simply-laced, and $\eta$ is minimal or next-to-minimal. In this case the output of the algorithm is analyzed in great detail in \cite{Part2}.
\end{enumerate}
\end{remark}

Let us now go to the other extreme and consider cuspidal $\eta$.

\begin{lem}\label{lem:cusp}
Let $\eta\in C^{\infty}(\Gamma\backslash G)$, and assume that the constant term $c_U(\eta):=\int_{[U]}\eta(u)du$ vanishes for any   $U\subset G$ which is a unipotent radical of a proper parabolic subgroup. Let $\cF_{S,\varphi}(\eta)$ be a non-vanishing Levi-distinguished Fourier coefficient. Then the orbit $\Gamma\varphi\in \fg^*$ is $\K$-distinguished.
\end{lem}
\begin{proof}
Recall that by Definition \ref{def:Levi-distinguished}, there is a decomposition $S=h+Z$ such that  $(h,\varphi)$ be a neutral Whittaker pair for $\fl:=\fg^Z$, the orbit of $\varphi$ in $\fl^*$ is $\K$-distinguished, and 
\begin{equation}\label{=LeviDistPf}
\fg^{h+Z}_{> 1}=\fg^{h+Z}_{\geq 2}=
\fg^{Z}_{>0}\oplus \fl^{h}_{\geq 2}\text{ and }\fg^{h+Z}_{ 1}=\fl^{h}_{1}.
\end{equation}
Let $\fp:=\fg^Z_{\geq 0}$, $P$ be the corresponding parabolic subgroup, and $U$ be the unipotent radical of $P$. By \eqref{=LeviDistPf}, $\cF_{S,\varphi}(\eta)=\cF_{h,\varphi}(c_U(\eta))$, where we view $c_U(\eta)$ as an element of $C^{\infty}(\Gamma\backslash G)$. Since $\cF_{S,\varphi}(\eta)$ does not vanish, neither does $c_U(\eta)$ and thus $P=G$. Thus $L=G$ and thus the orbit $\Gamma\varphi\in \fg^*$ is $\K$-distinguished.
\end{proof}

\section{Applications and examples} \label{sec:examples}

In this section we will illustrate how to apply the framework introduced in this paper to compute certain Fourier coefficients in detail. We begin in \S \ref{subsec:Heis} to consider the case when $G$ is split and simply-laced and $P\subset G$ is a parabolic subgroup with unipotent radical $U$ isomorphic to a Heisenberg group. We use Algorithm \ref{alg:domin} to express any automorphic function on $G$ in terms of its Fourier coefficients with  respect to $U$. In \S \ref{subsec:Whit3} we then give an example of a Whittaker triple and a quasi-Fourier coefficient for the group $G=\SL_4$. In \S \ref{subsec:GL} we demonstrate Algorithm \ref{alg:domin}, Corollary \ref{cor:algGood} and Remark \ref{rem:alg} for $G= \GL_n$. In \S\ref{subsec:Sp4} we demonstrate them for $G= \Sp_4$. 

In \cite{Part2} we apply Theorem \ref{thm:IntTrans}, Algorithm \ref{alg:domin}, and Proposition \ref{prop:Heis} below to small automorphic forms on all simple split simply-laced  groups.

As many examples below are built on classical groups, we shall use matrix notation and denote by $e_{ij}$ the elementary matrix with a $1$ at position $(i,j)$ and zeroes elsewhere.

\subsection{Fourier expansions along Heisenberg parabolics}\label{subsec:Heis}

Let $G$ be split and simply-laced, and let  $\fh\subset \fg$ be the Lie algebra of a maximal split torus.  Fix a choice of positive roots. For any simple root $\alp$ define $S_{\alp}\in \fh$ by $\alp(S_{\alp}):=2$ and $\beta(S_{\alp})=0$ for any other simple root $\beta$. 

\begin{defn}
We say that a simple root  $\alpha$ is a Heisenberg  root  if $\lie g^{S_\alpha}_{>0}$ is a Heisenberg Lie algebra, or, equivalently, if $g^{S_\alpha}_{4}$ has dimension one.
\end{defn}

\begin{lemma}
If $\fg$ is simple of type $A_n$, there are no Heisenberg roots. If $\fg$ is simple of type $D_n$ or $E_n$ then there exists a unique Heisenberg root, and this is the unique simple root satisfying $\langle \alp, \alp_{\max}\rangle=1,$ where $\alp_{\max}$ denotes the highest root. 
\end{lemma}
\begin{proof}
Let $\fg$ be simple and let $\alp$ be a Heisenberg root. Then $g^{S_\alpha}_{4}$ has to be the highest weight space of the adjoint representation, {\it i.e.} the root space of $\alp_{\max}$. Since $g^{S_\alpha}_{4}$ is one-dimensional, $\alp_{\max}-\beta$ is not a root for any simple root $\beta\neq \alp$. Thus $\alp_{\max}-\alp$ is a root, and thus $\langle \alp, \alp_{\max}\rangle=1$.
The roots $\alp$ with this property are precisely the nodes in the affine Dynkin diagram, that are connected to the affine node (corresponding to $-\alp_{\max}$).

Checking the affine Dynkin diagrams (see \cite[Tables IV-VII]{Bou4_6}), we see that there is a unique simple root $\alp$ with this property in types $D_n$ and $E_n$, and these roots are indeed Heisenberg. In the Bourbaki notation, these roots are $\alp_2$ for $D_n$ and $E_6$, $\alp_1$ for $E_7$ and $\alp_8$ for $E_8$. In type $A_n$  there are two roots with this property but none of them is Heisenberg. In fact,    $\lie g^{S_\beta}_{>0}$ is abelian for any simple root $\beta$ in type $A_n$. This is so, since in type $A_n$, $\alp_{\max}$ is the sum of all simple roots (with all coefficients being 1).
\end{proof}

\begin{notn}
Let $\alp$ be a 
root.   Define $h_{\alp}:=\alp^{\vee}\in \fh$  by requiring for all roots $\beta$ 
\begin{equation}\label{=h}
\beta(h_{\alp})=2\frac{\langle \alp,\beta\rangle}{\langle \alp,\alp \rangle}=\langle \alp,\beta\rangle\,.
\end{equation}
Denote also by $\fg^{\times}_{-\alp}$  the set of non-zero covectors in the dual root space $\fg^{*}_{-\alp}$.
\end{notn}
Note that for $\beta \neq \pm \alp,$ $\beta(h_{\alp})\in \{-1,0,1\}$.
By \cite[Proposition II.8.3]{Hum}, $(h_{\alp},\varphi)$ is a neutral pair for any $\varphi\in \fg^{\times}_{-\alp}$.

\begin{notn}
For any Heisenberg root $\alp$, let $\Omega_{\alp}\subset \Gamma$ be the abelian subgroup obtained by exponentiation of the abelian Lie algebra given by the direct sum of the root spaces of negative roots $\beta$ satisfying $\langle \alp, \beta\rangle=1$. Let 
\begin{equation}
\Psi_\alpha :=\{ \text{ root } \eps \mid \langle \eps, \alp \rangle \leq 0,  \eps(S_{\alp})=2\}.
\end{equation}
\end{notn}
Note that all the roots in $\Psi_{\alp}$ have to be positive.
  
In this subsection we use Algorithm \ref{alg:domin} to deduce the following proposition, that will be used in the sequel paper \cite{Part2}.

\begin{prop}\label{prop:Heis}
Let $\alp$ be a Heisenberg root. Let $\gamma_{\alp}\in \Gamma$ be a representative of a Weyl group element that conjugates $\alp$ to $\alp_{\max}$, where $\alp_{\max}$ denotes the maximal root of the component of $\fg$ corresponding to $\alp$. Then we have
\begin{equation}\label{=Heis}
\eta(g)=\sum_{\varphi\in (\fg^*)^{S_{\alp}}_{-2}}\cF_{S_{\alp},\varphi}[\eta](g)+\sum_{\varphi\in \fg^{\times}_{-\alp}}\sum_{\omega\in \Omega_{\alp}}\sum_{ \psi\in \bigoplus_{\eps\in \Psi_\alpha}\fg^*_{-\eps}}\cF_{S_{\alp},\varphi+\psi}[\eta](\omega \gamma_{\alp}g)\,.
\end{equation}
\end{prop}

For the proof we will need the following lemma.
\begin{lem}\label{lem:Heis}
Let $\alp$ be a Heisenberg root. Then for any $\varphi \in \fg^{\times}_{-\alp}$ we have 
\begin{equation}\label{=Heis2}
\cF_{h_{\alp},\varphi}[\eta](g)=\sum_{\omega\in \Omega_{\alp}}\sum_{ \psi\in \bigoplus_{\eps\in \Psi_\alpha}\fg^*_{-\eps}}\cF_{S_{\alp},\varphi+\psi}[\eta](\omega g)\,.
\end{equation}
\end{lem}
\begin{proof}
Following Algorithm \ref{alg:domin} we
consider the deformation $(1-t)h_{\alp}+tS_{\alp}$.
By Lemma \ref{lem:StvN}, we have 
\begin{equation}
\cF_{h_{\alp},\varphi}[\eta](g)=\sum_{\omega\in \Omega_{\alp}}\cF^{R_0}_{h_{\alp},\varphi}[\eta](\omega g)\,.
\end{equation}
Then, the critical values are $1/2$ and $2/3$, and the quasi-critical values are $1/3$ and $1$. 
At $1/3$,  we have no Whittaker triple entries yet and thus nothing  moves into the $-2$-eigenspace. At $1/2$, we get contributions  in the third component of the Whittaker triple from the root spaces of all the roots $\eps$ with $\langle \eps, \alp \rangle = 0$ and $\eps(S_\alp) =- 2$.
At $t=2/3$ we also get all the negative roots with 
$ \langle \eps, \alp \rangle = 1$ and $\eps(S_\alp) =- 2$. This means that we would get contributions from all these root spaces in the third component of the Whittaker triple. 
At $t=1$  the Whittaker triple becomes a Whittaker pair and thus we obtain
\begin{equation}
\cF_{h_{\alp},\varphi}[\eta](g)=\sum_{\omega\in \Omega_{\alp}}\cF^{R_0}_{h_{\alp},\varphi}[\eta](\omega g)=
\sum_{\omega\in \Omega_{\alp}}\sum_{ \psi\in \bigoplus_{\eps\in \Psi_\alpha}\fg^*_{-\eps}}\cF_{S_{\alp},\varphi+\psi}[\eta](\omega g)\,.
\end{equation}
\end{proof}
Since $\eta=\cF_{0,0}[\eta],$ in order to express $\eta$ in terms of Fourier coefficients of the form $\cF_{S_{\alp},\varphi}$ we need to consider the deformation $S_t:=tS_{\alp}$. To simplify the exposition we do that in more elementary terms.
\begin{proof}[Proof of Proposition \ref{prop:Heis}]
By the conditions, the Lie algebra $\fg^{S_{\alp}}_{>0}$ is a Heisenberg Lie algebra, with center $\fg^{S_{\alp}}_{4}$, and abelian quotient $\fg^{S_{\alp}}_{2}$.
We restrict $\eta$ to the exponential of the center and decompose to Fourier series. The constant term with respect to the  center $\fg^{S_{\alp}}_{4}$ is  $\cF_{S_{\alp}/3,0}[\eta]$, and the other terms  are $\cF_{S_{\alp}/2,\varphi}[\eta]$ for $\varphi\neq 0\in (\fg^*)^{S_{\alp}}_{-4} =(\fg^*)^{S_{\alp}/2}_{-2}$.  We remark that this constant term can be denoted $\cF_{cS_{\alp},0}[\eta]$ for any $1/4 \leq c<1/2$ but not for $c=1/2$ since $0$ defines a zero form on the 1-eigenspace, and thus $\fn_{cS_{\alp},0}=\fg^{cS_{\alp}}_{\geq 1}=\fg^{S_{\alp}}_{\geq c^{-1}}$ and $\fn_{S_{\alp}/2,0}=\fg^{S_{\alp}}_{\geq 2}$. Note also that $(\fg^*)^{S_{\alp}/2}_{-2}=\fg^{\times}_{-\alp_{\max}}$.
Altogether we have
\begin{equation}\label{=Fou1}
\eta(g)=\cF_{S_{\alp}/3,0}[\eta](g)+\sum_{\varphi\in \fg^{\times}_{-\alp_{\max}}}\cF_{S_{\alp}/2,\varphi}[\eta](g)\,.
\end{equation}

Note that $\gamma_{\alp}$ conjugates $S_{\alp}/2$ to $h_{\alp}$. Thus, by Lemma \ref{lem:conjugation-translation}, we have $$\cF_{S_{\alp}/2,\varphi}[\eta](g)=\cF_{h_{\alp},\Ad^*(\gamma_{\alp})\varphi}[\eta](\gamma_{\alp}g)$$ and  
\begin{equation}\label{=sTwist}
\sum_{\varphi\in \fg^{\times}_{-\alp_{\max}}}\cF_{S_{\alp}/2,\varphi}[\eta](g)=\sum_{\varphi\in \fg^{\times}_{-\alp}}\cF_{h_{\alp},\varphi}[\eta](\gamma_{\alp}g)\,.
\end{equation}

We restrict the constant term of \eqref{=Fou1} to the maximal abelian quotient of $\Exp(\fg^{S_{\alp}}_{>0})$, decompose to Fourier series and obtain  
\begin{equation}\label{=AbT}
\cF_{S_{\alp}/3,0}[\eta](g)=\sum_{\varphi\in (\fg^*)^{S_{\alp}}_{-2}}\cF_{S_{\alp},\varphi}[\eta](g)\,.
\end{equation}
Formula \eqref{=Heis} follows now from \eqref{=Fou1}, \eqref{=sTwist}, \eqref{=Heis2} and \eqref{=AbT}.
\end{proof}
\begin{remark}\label{rem:Heis}
Let us explain why we chose to use $S_{\alp}$ in Proposition \ref{prop:Heis}. In types $E_6,E_7,$ and $E_8$ this choice follows \S \ref{subsec:LeviDist}. Indeed,
the starting point is $\eta=\cF_{0,0}[\eta]$. In the first step we choose a generic element $z$ in the Cartan. Choose $z$ to be $2$ on the Heisenberg root $\alp$, and to be very small positive rational numbers on other roots. 
This obtained deformation gives the same results as the deformation with $Z=S_{\alp}$. The first critical value is $1/4$, at which we obtain the decomposition described in (\ref{=Fou1}). 
With the constant term we can proceed to the next critical value $1/2$, at which we obtain the decomposition in   (\ref{=AbT}). 

Then we conjugate the non-constant terms obtained in (\ref{=Fou1}) by $\gamma_{\alp}$. This is not part of the algorithm, but we do that for convenience. Now we need to choose a generic $z$ that commutes with the root space of $\alp$.
We choose it to be $1$ on the only simple root non-orthogonal to $\alp$, and very small positive rational numbers on other simple roots. The resulting decomposition appears in Lemma \ref{lem:Heis}. Altogether, this gives Proposition \ref{prop:Heis}.

To express $\eta$ in terms of its Whittaker coefficients one should continue with each of the terms in the right-hand side of \eqref{=Heis}. However, the obtained expression would be very long and complicated. In \cite{Part2} we provide the expression under the assumption that the Whittaker support of $\eta$ consists of the next-to-minimal orbit.

For groups of type $D_n$, \S \ref{subsec:LeviDist} would provide us with a different formula, but we still prove formula (\ref{=Heis}) for all cases for its uniformity, beauty, and future applications.
\end{remark}

\begin{remark} 
Here we elaborate a little on the structure of the Fourier expansion (\ref{=Heis}) and comment on the relation to previous works on Heisenberg expansions. The semisimple element $S_\alp$ defines a Heisenberg parabolic subgroup $P_\alp\subset G$ with Levi decomposition $P_\alp=LU$. The Lie algebra $\mathfrak{p}_\alp\subset \mathfrak{g}$ of $P_\alp$ exhibiting the following grading 
\begin{equation}
\mathfrak{p}_\alp=\mathfrak{g}_0\oplus \mathfrak{g}_1\oplus \mathfrak{g}_2,
\end{equation}
where the subscripts indicate the values of the inner products  $\langle \cdot , \alp_{\max}\rangle$. Thus $\mathfrak{g}_1$ is spanned by all roots $\eps$ such that $\langle \eps, \alp_{\max}\rangle=1$. Equivalently, these are all roots $\eps$ such that $\alp_{\max}-\eps$ is also a root. Notice that all roots in $\mathfrak{g}_1$ are  positive  and the only simple root satisfying the condition $\langle \eps, \alp_{\max}\rangle=1$ is $\alpha$ itself. Since only $\eps=\alpha_{\max}$ satisfies $\langle \eps, \alp_{\max}\rangle=2$ the space $\mathfrak{g}_2$ is one-dimensional, spanned by $E_{\alp_{\max}}$. The zeroth subspace $\mathfrak{g}_0$ is the Lie algebra of the Levi $L\subset P$. The subspace $\mathfrak{g}_1\oplus \mathfrak{g}_2$ is thus the Heisenberg nilpotent subalgebra with center $\mathfrak{g}_2$. Notice that $\sum_{\gamma\in \Psi_\alpha}\mathfrak{g}_\gamma$ is a Lagrangian subspace of $\mathfrak{g}_1$. Indeed, $\mathfrak{g}_1$ has a canonical Lagrangian decomposition 
\begin{equation}
\mathfrak{g}_1=\sum_{\gamma\in \Psi_\alpha}\mathfrak{g}_\gamma \oplus \sum_{\gamma\in \Psi^{\perp}_\alpha}\mathfrak{g}_\gamma,
\end{equation}
where $\Psi_\alp^{\perp}$ is the orthogonal complement 
\begin{equation}
\Psi_\alp^{\perp}=\{\text{ root } \eps \mid \langle \eps, \alp \rangle \geq  1, \langle \eps, \alp_{\max}\rangle=1 \}.
\end{equation}
Note that the root $\alpha$ belongs to $\Psi_\alp^{\perp}$. The Fourier expansion (\ref{=Heis}) thus corresponds to the standard  non-abelian Fourier expansion along the Heisenberg unipotent $U$, which exhibits a sum over the center $\mathfrak{g}_2$ along with a sum over a Lagrangian subspace $\Psi_\alp$ of $\mathfrak{g}_1$. The choice of Lagrangian decomposition is usually referred to as a choice of ``polarization''. Similar kinds of expansions have been treated in several places in the literature; see \cite{MR1159103,Kazhdan:2001nx,PP,BKNPP,FGKP} for a sample. In the notation of the original paper by Kazhdan and Savin \cite{MR1159103}, the space $\Psi_\alp$ corresponds to $\Pi_o^{*}$ while $\Psi_\alp^{\perp}$ corresponds to $\Pi_o$.
\end{remark}

\subsection{Whittaker triples}\label{subsec:Whit3}
We will now illustrate what type of quasi-Fourier coefficients we are able to describe using Whittaker triples that are not captured by Whittaker pairs in an example for $G = \SL_4$.

Let $(S, \varphi, \psi)$ be the Whittaker triple with $S = \tfrac{1}{3}\diag(3,1,-1,-3)$, $\varphi = e_{41}$ and $\psi = m e_{31} + n e_{42}$, where $m,n \in \K$ and $e_{ij}$ denote elementary matrices. The $S$-eigenvalues for the different elementary matrices can be illustrated by the following matrix
\begin{equation}
    \begin{psmallmatrix}
          0 & 2/3   & 4/3   & 2   \\
       -2/3 & 0     & 2/3   & 4/3   \\
       -4/3 & -2/3  & 0     & 2/3 \\
       -2   & -4/3  & -2/3  & 0    \\
    \end{psmallmatrix} \, ,
\end{equation}
from which we may read out that $\varphi$ has eigenvalue $-2$ while $\psi$ has eigenvalue $-4/3$.

As seen from this matrix we get the following unipotent subgroup (independent of $\psi$)
\begin{equation}
    N_{S,\varphi} = \left\{
        \begin{psmallmatrix}
            1 & 0 & x_2 & x_1 \\
              & 1 & 0 & x_3 \\
              &   & 1 & 0 \\
              &   &   & 1
        \end{psmallmatrix} : x_1, x_2, x_3 \in \mathbb{A} \right\} \, ,
\end{equation}
and the corresponding Fourier coefficient of an automorphic function $\eta$  can be expressed as
\begin{equation}
    \mathcal{F}_{S,\varphi,\psi}[\eta](g) = \intl_{(\K\bs\A)^3} \eta\Bigl( 
    \begin{psmallmatrix}
        1 & 0 & x_2 & x_1 \\
          & 1 & 0 & x_3 \\
          &   & 1 & 0 \\
          &   &   & 1
    \end{psmallmatrix}g\Bigr) {\chi(x_1 + m x_2 + n x_3)}^{-1} \, d^3x \, ,
\end{equation}
where we recall that $\chi$ is a fixed non-trivial character on $\mathbb{A}$ trivial on $\mathbb{K}$.

From this example we see that we require Whittaker triples in addition to Whittaker pairs if we want to construct Fourier coefficients with characters that are not only supported on $x_1$ but also on $x_2$ and $x_3$.

\subsection{The case of \texorpdfstring{$\GL_n$}{GL(n)}}\label{subsec:GL}
Let $\bf G:=\GL_n$, $G:=\GL_n(\A)$, and $\Gamma:=\GL_n(\K)$. In this section we will follow Algorithm~\ref{alg:domin} and \S \ref{subsec:LeviDist} to present any automorphic function $\eta\in C^{\infty}(\Gamma \backslash G)$ as a countable linear combination of its Whittaker coefficients. We will show that our proof amounts in this case morally to the same decomposition as in \cite{PiatetskiShapiro,Shalika,JiangLiu}.

In \cite{PiatetskiShapiro,Shalika}, $\eta$ is first restricted to the mirabolic nilradical, {\it i.e.} 
\begin{align}
U=\left \{\left(\begin{array}{cc}
     \Id_{n-1} & *  \\
     0 &1\\
   \end{array}\right) \right\},
\end{align}
and decomposed into Fourier series with respect to $U$. 
Our algorithm does the same thing, but in several steps. 
First let $(h,\varphi)=(0,0)$. Let $N\gg0$, 
\begin{align}
z_1:=\diag(0,-1,-N,\dots,-N^{n-3},-N^{n-2}),
\end{align}
and consider the deformation $S_t:=tz_{1}$. Under this deformation, the first thing that happens is that the highest root space (spanned by $e_{1n}$) enters $\fg^{S_t}_1$. At this point $\eta$ decomposes into a sum of quasi-Fourier coefficients. At the next step $e_{1n}$ enters $\fg^{S_t}_2$, and the quasi-Fourier coefficients become Fourier coefficients. For the constant term, we continue with the same deformation, until $e_{2n}$ enters. 
For the non-constant term we have to change the deformation into something that will commute with $\varphi$. The $\varphi$ can be identified with $ae_{n1}$ under the trace form, for some $a\in \K^{\times}$. We take the deformation by 
\begin{align}
z_2:=\diag(-N^{n-2},-1,-N,\dots,-N^{n-4},-N^{n-3},-N^{n-2}),
\end{align}
and continue in the same way.  Eventually, all of $U$ enters and all possible characters (including the trivial one) appear. 

Let us now analyze the summands. The constant term is $\cF_{tz_1,0}$ for $t=2/(N^{n-2}-N^{n-3})$, and we can continue the deformation along $z_1$. Any non-trivial character of $U$ can be conjugated using $\GL_{n-1}$ (embedded into the upper left corner) to the one given by $e_{n,n-1}$. We can now choose the deformation 
\begin{equation}
z_3:=\diag(-1,-N,\dots,-N^{n-4},-N^{n-3},-N^{n-3}).
\end{equation}
In the same way as above, it will give a decomposition of $\cF_{tz_1,0}$ into Fourier series with respect to the column $n-1$, \ie
\begin{equation}
U'=\left \{\left(\begin{array}{ccc}
     \Id_{n-2} & *  & 0 \\
     0 &1&0\\
     0&0&1
   \end{array}\right) \right\}.
\end{equation}
   
Continuing in this way we obtain
\begin{equation}\label{=PSS}
\eta(g)=\sum_{x\in 2^{[n-1]}}\sum_{\gamma\in \Gamma_x}\cF_{S,\varphi_x}[\eta](\gamma g),
\end{equation}
where $[n-1]$ denotes the set $\{1,\dots,n-1\}$, $2^{[n-1]}$ denotes the set of all its subsets, $S=\diag(n-1,n-3,\dots, 3-n,1-n)$, and for any $x\in 2^{[n-1]}$,  $\psi_x:=\sum_{i\in x}e_{i+1,i}$ and $\Gamma_x$ is a certain subset of $\Gamma$. 

For cuspidal $\eta$ and $x\neq [n-1]$, we have $\cF_{S,\varphi_x}[\eta]=0$ and \eqref{=PSS} becomes the formula in \cite{PiatetskiShapiro,Shalika}.
For $\eta$ in the discrete spectrum, $\cF_{S,\varphi_x}[\eta]$ vanishes for many $x$ by \cite[Lemma 3.2]{JiangLiu}, and 
\eqref{=PSS} reflects the formula in \cite[Theorem 3.3]{JiangLiu}.
If $\eta$ is minimal then $\cF_{S,\varphi_x}[\eta]=0$ for $|x|>1$ and if $\eta$ is next-to-minimal then $\cF_{S,\varphi_x}[\eta]=0$ for $|x|>2$. 
These cases were computed in  \cite{Ahlen:2017agd}, motivated by applications in string theory.

\subsection{Examples for \texorpdfstring{$\Sp_{4}$}{Sp(4)}}\label{subsec:Sp4}
Let $G:=\Sp_4(\A)$, $\Gamma:=\Sp_4(\K)$ and let $\eta\in C^{\infty}(\Gamma\backslash G)$. In this section we express $\eta$ in tquaserms of its Levi-distinguished Fourier coefficients, providing an example for Algorithm~\ref{alg:domin} and Remark \ref{rem:alg}.
Let $\fg:=\Lie(\Gamma)$, realized in $\gl_4$ by the $2\times 2$ block matrices \begin{equation}\label{=block}
\left(
   \begin{array}{cc}
     A & B=B^t  \\
     C=C^t & -A^t\\
   \end{array}
 \right).
 \end{equation}
 
Let $\fn \subset \fg$ be the maximal unipotent subalgebra spanned by the matrices $e_{12}-e_{43}, e_{13}, e_{24}, e_{14}+e_{23}$ and let $N:=\Exp(\fn(\A))$. For any $a,b\in \K$ denote by $\chi_{a,b}$ the character of $\fn$ given by $\chi_{a,b}(e_{12}-e_{43})=a$ and $\chi_{a,b}(e_{24})=b$, and let $\cW_{a,b}$ denote the corresponding Whittaker coefficient. 
Let $\fu\subset \fn$ be the Siegel nilradical, \ie the normal commutative subalgebra spanned by the matrices $e_{13}, e_{24}, e_{14}+e_{23}$ and let $U:=\Exp(\fu(\A))$. Let $L$ denote the Siegel Levi subgroup of $\Gamma$ given by $\diag(g,(g^t)^{-1})$, where $g\in \GL_2(\A)$. Using the trace form on $\fg$, we can identify $\fu^*$ with the nilradical $\bar \fu$ of the opposite parabolic, \ie with the space of matrices of the form \eqref{=block} with $A=B=0$.
Note that $\bar \fu\cong \Sym^2(\K^2)$, and $L$ acts on it by the standard action on symmetric forms.  
For any $\varphi \in \fu^*\cong \bar \fu \cong \Sym^2(\K^2)$, denote by $\cF_{\fu,\varphi}$ the corresponding parabolic Fourier coefficient.

Let us now outline the strategy for this subsection.
According to \S \ref{subsec:LeviDist} we should choose a generic element $z$ of the Cartan. Choose it to be 1 on the long simple root $\alp_2=2\eps_2$ that defines $U$ and a small positive rational number on the short simple root. Then, the first steps of the algorithm will provide us with decomposition of $\eta$ into Fourier series along the abelian unipotent radical $U$. The coefficients will be parameterized by characters of $\fu$, that can be identified with quadratic forms on $\K^2$. For the zero form we can just continue the deformation all the way to a Whittaker coefficient. All forms of rank one are conjugate, so we can conjugate them to a convenient form and again continue the deformation, obtaining a Whittaker coefficient. The same goes for split forms of rank two. The non-split forms of rank two belong to $\K$-distinguished orbits, and thus the corresponding Fourier coefficients are already Levi-distinguished. 

We will now give all the details of the decompositions we just mentioned. We will do this in elementary terms in a self-contained way.

Since $U$ is abelian, the Fourier decomposition on it gives 
\begin{equation}\label{=SpUFou}
    \eta=\sum_{\varphi\in \fu^*}\cF_{\fu,\varphi}[\eta]\,.
\end{equation}

We now decompose this sum into three different terms, by the rank of $\varphi$, viewed as a quadratic form. Let us first analyze the constant term $\cF_{\fu,0}[\eta]$.
We restrict it to $L$, and decompose to Fourier series on the abelian group $N\cap L$. We obtain 
\begin{equation}
\cF_{\fu,0}[\eta]=\sum_{a\in \K}\cW_{a,0}[\eta]\,.
\end{equation}

Next, any $\varphi$ of rank one is conjugate under $L$ to 
$\varphi_1:=
\begin{psmallmatrix}
    1 & 0  \\
    0 & 0\\
\end{psmallmatrix}$. This $\varphi_1$ is normalized by $N$, and thus we can again decompose $\cF_{\fu,\varphi}[\eta]$ on $N\cap L$. We obtain 
\begin{equation}
    \cF_{\fu,\varphi_1}[\eta]=\sum_{a\in \K}\cW_{a,1}[\eta]\,.
\end{equation}

The non-degenerate forms (\ie those of rank two) can be divided into two subsets: split and non-split. All the split ones are conjugate under $L$ to $\varphi_2:=
\begin{psmallmatrix}
    0 & 1  \\
    1 & 0\\
\end{psmallmatrix}$. Let $w\in \Gamma$ denote a representative for the Weyl group element given by the simple reflection with respect to the long simple root $\alp_2=2\eps_2$, \eg $w=\diag(1,1,1,-1)\sigma_{24}$, where $\sigma_{24}$ is the permutation matrix on indices 2 and 4.
Then $\fu^w=\Span(e_{12}-e_{43}, e_{13}, e_{42})$, and $\varphi_2^w$ equals the restriction to $\fu^w$ of $\chi_{1,0}$. Using Corollary \ref{cor:RootExchange}, we can express $\cF_{\fu^w,\chi_{1,0}}$ through $\cF_{\fu',\chi_{1,0}}$, where $\fu'=\Span(e_{12}-e_{43}, e_{13}, e_{24})\subset \fn$. The integration will be over elements matrices of the form $v_x:=\Id+xe_{24}\in G$.
Using Fourier expansion by the remaining coordinate of $e_{14}+e_{23}\in \fn$, we obtain
\begin{equation}\label{=rk2}
\cF_{\fu,\varphi}[\eta](g)=\int\limits_{x}\cW_{1,a}[\eta](v_{x}wg) dx.
\end{equation}
Finally, let $X\subset \bar u \cong \fu^*$ denote the set  of anisotropic non-degenerate forms. For $\varphi \in X$, we have no expression of $\cF_{\fu,\varphi}[\eta]$ in terms of Whittaker coefficients. However, any $\varphi\in X$ is  $\K$-distinguished. Indeed, let $h:=\Id \in \fl$.  Then $(h,\varphi)$ is a neutral pair, and its centralizer is anisotropic. By Lemma \ref{lem:K-distinguished} applied to  $(h,\varphi)$  and $Z:=0$,  $\varphi$ is $\K$-distinguished. 

Combining \eqref{=SpUFou}--\eqref{=rk2} we obtain the following theorem, that exemplifies Algorithm~\ref{alg:domin} and Remark \ref{rem:alg}.

\begin{thm}\label{thm:Sp}
For any $\eta\in C^{\infty}(\Gamma \backslash G)$ and $g\in G$, $\eta(g)$ equals
\begin{equation}
\sum_{\varphi\in X}\cF_{\fu,\varphi}[\eta](g)+\sum_{a\in \K}\Biggl(\sum_{\gamma\in L/ O(1,1)}\int\limits_{\A} \cW_{1,a}[\eta](v_xw\gamma g)dx+\suml{\gamma\in L/(N\cap L)}\cW_{a,1}[\eta](\gamma g)+\cW_{a,0}[\eta](g)\Biggr),
\end{equation}
where $O(1,1)\subset L$ denotes the stabilizer of the split form $\varphi_2$.
\end{thm}

If $\eta$ is cuspidal then $\cW_{0,a}[\eta]=\cW_{a,0}[\eta]=0$. If $\eta$ is non-generic $\eta$, then $\cW_{1,a}[\eta]=\cW_{a,1}[\eta]=0$, unless $a=0$. Thus Theorem \ref{thm:Sp} implies the following corollary.

\begin{cor}\label{cor:Sp}
Let $\eta\in C^{\infty}(\Gamma \backslash G)$ and $g\in G$. 
\begin{enumerate}
\item  \label{it:SpCusp}If $\eta$ is cuspidal then
\begin{equation}
\eta(g)=\sum_{\varphi\in X}\cF_{\fu,\varphi}[\eta](g)+\sum_{a\in \K^{\times}}\Biggl(\sum_{\gamma\in L/ \mathrm{O}(1,1)}\int\limits_{\A} \cW_{1,a}[\eta](v_xw\gamma g)dx +\suml{\gamma\in L/(N\cap L)}\cW_{a,1}[\eta](\gamma g)\Biggr)\,.
\end{equation}
\item \label{cor:Sp2} If $\eta$ is non-generic then 
\begin{equation}
\eta(g)=\sum_{\varphi\in X}\cF_{\fu,\varphi}[\eta](g)+\hspace{-1em}\sum_{\gamma\in L/ \mathrm{O}(1,1)}\int\limits_{\A}\!\cW_{1,0}[\eta](v_xw\gamma g)dx +\suml{\gamma\in L/(N\cap L)}\cW_{0,1}[\eta](\gamma g)+\sum_{a\in \K}\cW_{a,0}[\eta](g)\,.
\end{equation}
\item If $\eta$ is cuspidal and non-generic then 
$\eta=\sum_{\varphi\in X}\cF_{\fu,\varphi}[\eta]$.
\end{enumerate}
\end{cor}

\appendix

\section{On PL-orbits}\label{subsec:PL}

\numberwithin{lemma}{section}

A complex orbit is a PL-orbit if and only if its Bala-Carter label has no parenthesis. In particular, all complex minimal and next-to-minimal orbits are PL. The classification of PL orbits of complex classical groups in terms of the corresponding partitions is given in \cite[\S 6]{GS15}.

The classification of rational PL-orbits is a more complicated task. 
In this subsection we discuss the PL property for small $\K$-rational orbits of simple split groups.
A complex orbit $\cO_{\C}$ may include several or even infinitely many rational orbits. If $\cO_{\C}$ is non-PL then all its rational orbits are non-PL. If $\cO_{\C}$ is PL then it includes at least one rational PL-orbit, but can also include non-PL rational orbits. In type $A_n$, all rational orbits are PL. Let us now describe the PL properties of minimal and next-to-minimal orbits. Here, minimal and next-to-minimal refers to the closure order on the complex orbits, which might be coarser than the order defined in Definition \ref{def:order}.

All minimal rational orbits are PL. Indeed, for classical groups it is easy to establish the Levi in which they are principal: for $\SO_{n+1,n}$ it is $\SO_{2,1}\times (\GL_1)^{n-1}$, for $\Sp_{2n}$ it is $\Sp_2\times (\GL_1)^{n-1}$ and for $\SO_{n,n}$ it is $\SO_{2,2}\times (\GL_1)^{n-2}$. For exceptional groups, the rational minimal orbit is unique and thus PL. This uniqueness was explained to us by Joseph Hundley. 
Let us now deal with the next-to-minimal orbits.
\begin{lemma}
    All next-to-minimal rational orbits for $\SO_{n,n}$ and $\SO_{n+1,n}$ are PL.
\end{lemma}
\begin{proof}
    One can give a the classification of the rational orbits in the spirit of the classification of real orbits given in \cite[\S 9.3]{CM}. 
    Namely, a $\K$-rational orbit with a given partition is defined by a collection of quadratic forms $Q_{2i+1}$ on multiplicity spaces of the odd parts. If we add a hyperbolic form to the direct sum of these forms we get the initial form, which is also hyperbolic. Here, a hyperbolic form is a direct multiple of the 2-dimensional quadratic form given by $H(x,y)=xy$.
    By Witt's cancelation theorem this implies that the direct sum of the forms on multiplicity spaces of the odd parts is hyperbolic. 

    An orbit for $\SO_{n,n}$ is PL if and only if all $Q_{2i+1}$ are hyperbolic, except $Q_{2j+1}$ for a single index $j\geq 1$, which is a direct sum of a hyperbolic form and a one-dimensional quadratic form. For $\SO_{n,n}$ there are two next-to-minimal partitions.
    One of them is $2^41^{2n-8}$. For it, $Q_1$ has to be hyperbolic.
    The other next-to-minimal partition is $3 1^{2n-3}$. Thus $Q_3$ is one-dimensional. Now, note that $H^n=Q_3\oplus-Q_3 \oplus H^{n-1}$. Thus, $Q_3\oplus Q_1=Q_3\oplus-Q_3 \oplus H^{n-1}$ and thus $Q_1=(-Q_3) \oplus H^{n-1}$, \ie $Q_1$ is a direct sum of a hyperbolic form and a one-dimensional quadratic form.

    Similarly, it is easy to see that the next-to-minimal orbits for $\SO_{n+1,n}$\ are principal in Levis isomorphic to $(\GL_2)^2\times (\GL_1)^{n-4}$ or  $\SO_{2,1}\times (\GL_1)^{n-1}$. 
\end{proof}

However, $\Sp_{2n}(\K)$ has infinitely many rational next-to-minimal orbits, already for $n=2$.  Moreover, there exist cuspidal next-to-minimal representations of $\Sp_4(\A)$. Note that cuspidal non-generic automorphic forms cannot expressed through their Whittaker coefficients, since the latter coefficients have to vanish on such forms. See \cite[\S 4]{Ginz} for a discussion of cuspidal representations, in particular those of $\Sp_4(\A)$. 

As for the exceptional groups, Joseph Hundley showed that the next-to-minimal orbit is unique, and thus PL, for $E_6,E_7,E_8$ and $G_2$~\cite{HundleyMail}.

The group $F_4$ has infinitely many rational next-to-minimal orbits.
We expect that infinitely many of them are not PL.

\section{Some geometric lemmas}\label{sec:Geo}
\setcounter{lemma}{0}

\begin{lemma}\label{lem:SameOrbit}
    Let $Z\in \fg$ be rational semi-simple, let $\varphi\in \fg^Z_0$ and $\varphi'\in \fg^Z_{>0}$. Assume that $\varphi$ is conjugate to $\varphi+\varphi'$ by ${\bf G}(\C)$. Then there exist $X\in \fg^{Z}_{>0}$ such that $\ad^*(X)(\varphi)=\varphi'$ and $v\in \Exp(\fg^{Z}_{>0})$ such that $\Ad^*(v)(\varphi)=\varphi+\varphi'$. 
\end{lemma}

\begin{proof}
    Decompose $\varphi'=\sum_{i=1}^k\varphi'_i$ where $\varphi'_i\in (\fg^*)^{Z}_{\lam_i}$ and $\lam_1<\lam_2<\dots <\lam_k\in \Q_{>0}$ are all the positive eigenvalues of $Z$.

Let us first construct $X$. For any $t\in \R$, we have the following identity in $\fg^{*}(\C)$:
    \begin{equation}
        \Ad^*(\exp(tZ))(\varphi+\varphi')=\varphi+\sum_{i=1}^k \Ad^*(\exp(t\lam_i))\varphi'_i\,.
    \end{equation}
    Thus, $\varphi + \sum_i \Ad^*(\exp(t\lam_i))\varphi'_i\in G(\C)\varphi$. 
    Differentiating by $t$ at $0$ we obtain that $\sum_i \lam_i\varphi'_i$
    lies in the tangent space to the orbit ${\bf G}(\C)\varphi$ at $\varphi$. This tangent space is the image of $\varphi$ under the coadjoint action. Thus there exists $Y_{\C}\in \fg(\C)$ with $\ad^*(Y)(\varphi)=\sum_i \lam_i\varphi'_i$.
Since both $\varphi$ and $\sum_i \lam_i\varphi'_i$ lie in the $\K$-points $\fg^*$, there exists $Y\in \fg$ with the same property. Decompose $Y=Y'+\sum_i Y_i$ with $Y_i\in \fg^Z_{\lam_i}$. Since $\varphi$ commutes with $Z$, we obtain $\ad^*(Y_i)(\varphi)=\lam_i\varphi'_i$. Now we take $X:=\sum_i \lam_i^{-1}Y_i \in \lie g^Z_{>0}$. 

We now prove the existence of $v$ by descending induction on the maximal index $i$ such that $\varphi'\in \fg^Z_{>\lam_i}$. The base case $i=k$ has $\varphi'=0$. For the induction step, let $i<k$ such that $\varphi'\in \fg^Z_{>\lam_i}$.
Then $\Ad^*(\exp(-X))(\varphi+\varphi')=\varphi+\psi$, where $\psi\in \fg^Z_{>\lam_{i+1}}$. By the induction hypothesis, $\varphi+\psi\in \Ad^*(\Exp(\fg^{Z}_{>0}))\varphi$.
\end{proof}

\begin{cor}\label{cor:complexDim}
    Let $\cO,\cO'$ be two nilpotent $\Gamma$-orbits with $(\cO, \cO')\in R$ (see Definition \ref{def:order}), and let $\cO_{\C},\cO'_{\C}$ denote their complexifications. Then $\dim \cO_{\C}<\dim \cO'_{\C}$. 
\end{cor}
\begin{proof}
    By Lemma~\ref{lem:orbit-closure} we have $\cO_{\C}\subset \overline{\cO'_{\C}}$. Thus either 
   $\dim \cO_{\C}<\dim \cO'_{\C}$ or $ \cO_{\C}= \cO'_{\C}$.
    If $ \cO_{\C}= \cO'_{\C}$ then, by the definition of $R$, there exist a rational semi-simple $Z\in \fg$, $\varphi\in \cO\cap \fg^Z_0$,  and $\psi \in \fg^Z_{>0}$ such that $\varphi+\psi\in \cO_{\C}$, but $\varphi+\psi\notin \cO$. This contradicts Lemma \ref{lem:SameOrbit}.
\end{proof}

\begin{cor}\label{cor:order}
    The relation $R$ of Definition \ref{def:order} is indeed an order relation.
\end{cor}
\begin{proof}
    We have to show that if $(\cO,\cO')\in R$ then $(\cO',\cO)\notin R$. Suppose the contrary. Then by Lemma~\ref{lem:orbit-closure}  the complexifications  $\cO'_{\C}$ and $\cO_{\C}$ coincide. Moreover, because of the above assumption there exist a rational semi-simple $Z\in \fg$, $\varphi\in \cO\cap \fg^Z_0$,  and $\psi \in \fg^Z_{>0}$ such that $\varphi+\psi\in \cO_{\C}$, but $\varphi+\psi\notin \cO$. This contradicts Lemma \ref{lem:SameOrbit}.
\end{proof}

\begin{lemma}\label{lem:SameOrbit2}
    Let $Z,S\in \fg$ be commuting rational semi-simple elements, let $q\in \Q$ and let    
    $\varphi\in \fg^Z_0\cap \fg^{S}_{q}$ and $\varphi'\in \fg^Z_{>0}\cap \fg^{S}_{q}$. Assume that $\varphi$ is conjugate to $\varphi+\varphi'$ by ${\bf G}(\C)$. Then there exist $X\in \fg^{Z}_{>0}\cap \fg^{S}_{0}$ such that $\ad^*(X)(\varphi)=\varphi'$ and $v\in \Exp(\fg^{Z}_{>0}\cap \fg^{S}_{0})$ such that $\Ad^*(v)(\varphi)=\varphi+\varphi'$. 
\end{lemma}
\begin{proof}
To construct $X$ we proceed in the same way as in the proof of Lemma \ref{lem:SameOrbit}, and then decompose it with respect to eigenspaces of $S$ and take projection on the $0$ eigenspace. Then we construct $v$  in the same way as in the proof of Lemma \ref{lem:SameOrbit}.
\end{proof}
Lemma \ref{lem:SameHOrbit} follows now from Lemma \ref{lem:SameOrbit2} (with $q=2$) and Corollary  \ref{cor:complexDim}.

We would now like to relate the notion of dominance to dimensions. For any Whittaker pair $(S,\varphi)$
 define 
\begin{equation}
    d(S,\varphi):= (\dim \fn_{S,\varphi}+ \dim (\fg^S_{\geq 1}))/2
\end{equation}

\begin{example}
If $(S,\varphi)$ is a neutral pair then $d(S,\varphi)=(\dim \Gamma \varphi)/2$. If $(S,\varphi)$ is a Levi-distinguished pair corresponding to a nilpotent orbit $\cO$ in a Levi subalgebra $\fl\subset \fg$ then $d(S,\varphi)=(\dim \fg - \dim \fl +\dim \cO )/2$.
Furthermore, if $\fg^{S}_1 = \{0\}$, as in for example a parabolic Fourier coefficient, then $d(S,\varphi) = \dim \fn_{S,\varphi}$.
\end{example}

\begin{lem}
  \label{lem:dim-max-iso}
The number $d(S,\varphi)$ equals the dimension of any maximal isotropic subspace of $\fu_S:=\fg^S_{\geq 1}$. 
\end{lem}
\begin{proof}
Any such subspace includes $\fn_{S,\varphi}$, and the quotient is Lagrangian in the symplectic space $ \fu_S/\fn_{S,\varphi}$.
\end{proof}

We study Fourier coefficients with respect to maximal isotropic subspaces, so-called Fourier--Jacobi coefficients, in \cite{Eulerianity}.

\begin{lemma}\label{lem:dim}
Let $(H,\varphi)$ and $(S,\varphi)$ be Whittaker pairs with the same $\varphi$.  
\begin{enumerate}
    \item \label{it:DimDom} If $(H,\varphi)$ dominates $(S,\varphi)$ then  $d(H,\varphi)\leq d(S,\varphi)$.
    \item \label{it:DimLeviDist} If $(S,\varphi)$ is Levi-distinguished then $\dim \fn_{H,\varphi}\leq \dim \fn_{S,\varphi}$.
\end{enumerate}
\end{lemma}

\begin{proof}
For part \ref{it:DimDom} let $Z:=S-H$ and choose a Lagrangian subspace $a\subset (\fg^H_1\cap \fg^Z_{0})/( \fg^H_1\cap \fg^Z_{0}\cap \fg_{\varphi})$. 
Let $a'$ denote the preimage of $a$ in $\fg^H_1\cap \fg^Z_{0}$.
For any rational $t\in [0,1]$ denote $H_t:=H+tZ,$ and define  $\fl_t$ and $\fr_t$ as in~\eqref{=lt}. 
Define also $\fl_t^\text{max}:=\fl_t+a'$ and $\fr_t^\text{max}:=\fr_t+a'$. 
Then both $\fl^\text{max}_t$ and $\fr^\text{max}_t$ are maximal isotropic subspaces in $\fu_{H_t} := \fg^{H_t}_{\geq1}$ with respect to the anti-symmetric form $\omega_\varphi : \fg \times \fg \to \mathbb{K}$ defined by $\omega_\varphi(X, Y) = \varphi([X,Y])$. 
Indeed, the symplectic space $ \fu_{H_t}/\fn_{H_t,\varphi}$ is naturally isomorphic to $\fw_t/(\fw_t\cap \fg_\varphi)$, where $\fw_t:=\fg^{H_t}_1$.  Note that $(\fw_t)^Z_0=\fg^H_1\cap \fg^Z_{0}$ for all $t$.
Now, $\fw_t=(\fw_t)^{Z}_0\oplus (\fw_t)^{Z}_{<0}\oplus (\fw_t)^Z_{>0}$, with $(\fw_t)^Z_{<0}$ and $(\fw_t)^Z_{>0}$ both isotropic and orthogonal to $(\fw_t)^Z_{0}$ with respect to $\omega_\varphi$. 
Let $\fw_t^+$ and $\fw_t^-$ denote the images of $(\fw_t)^Z_{<0}$ and $(\fw_t)^Z_{>0}$ in $\fw_t/(\fw_t\cap \fg_\varphi)$.

Since $\fl^\text{max}_t$ projects onto $\fw_t^{-} \oplus a$ and $\fr^\text{max}_t$ projects onto $\fw_t^{+} \oplus a$, both project to Lagrangian subspaces of $\fu_{H_t}/\fn_{H_t,\varphi}\cong\fw_t/(\fw_t\cap \fg_\varphi),$
and thus are maximal isotropic. 
Hence, by Lemma \ref{lem:dim-max-iso}, we have
\begin{equation}
(H_t,\varphi)=\dim \fl_t^\text{max}=\dim \fr_t^\text{max}.
\end{equation}
Now let $0=t_0,\dots,t_n=1$ be all the critical numbers in the interval $[0,1]$. 
Then by \eqref{=lrw} for every $i$  we have $\fr_{t_i}\subseteq \fl_{t_{i+1}},$ thus $\fr^\text{max}_{t_i}\subseteq \fl^\text{max}_{t_{i+1}},$ and thus $d(H_{t_i},\varphi)\leq d(H_{t_{i+1}},\varphi)$. 
Since $H_{t_0}=H$ and $H_{t_n}=S$ part \ref{it:DimDom} follows.

For part \ref{it:DimLeviDist}, let $Z'$ be as in Notation \ref{not:Zprime}, and let $S':=H+TZ'$ with $T$ large enough as in Lemma \ref{lem:LeviDist}.
Then $\fg^{S'}_1 \subseteq \fg^H_1$ and thus 
\begin{equation}
d(H,\varphi)-\dim \fn_{H,\varphi}\leq d(S',\varphi)-\dim \fn_{S',\varphi}.
\end{equation}
Furthermore, by Lemma \ref{lem:Z'main}, $(H,\varphi)$ dominates $(Z',\varphi)$. Part \ref{it:DimDom} implies now that $\dim \fn_{(H,\varphi)}\leq \dim \fn_{(S',\varphi)}$. 
Finally, by Lemma \ref{lem:LeviDist} the pair $(S',\varphi)$ is Levi-distinguished, and
thus, by Lemma \ref{lem:SameDim}, we have 
$\dim \fn_{S,\varphi}= \dim \fn_{S',\varphi}$.
\end{proof}

\begin{remark}
In a previous arXiv preprint version of this paper we claimed that if $(H,\varphi)$ dominates $(S,\varphi)$ then 
$\dim \fn_{H,\varphi}\leq \dim \fn_{S,\varphi}$, and that therefore~$\dim \fn_{H,\varphi}$ is minimal for a neutral pair~$(H, \varphi)$. Unfortunately, these statements are wrong. Indeed, it is possible that two pairs $(H,\varphi)$ and $(S,\varphi)$ dominate each other, with $(H,\varphi)$ neutral, and $\dim \fn_{H,\varphi} > \dim \fn_{S,\varphi}$. 
\end{remark}

\begingroup\raggedright\endgroup


\end{document}